\newcommand{\kk}{\Bbbk}
\newcommand{\HH}{\mathrm{H}}
\newcommand{\Si}{\mathfrak{S}}
\newcommand{\Ext}{{\mathrm{Ext}}}

\newcommand{\PP}{\mathcal{P}}

\documentclass{amsart}

\usepackage{amssymb}

\usepackage[cmtip,all]{xy}

\usepackage{}

\newtheorem{theorem}{Theorem}[section]
\newtheorem{mproblem}[theorem]{Main Problem}
\newtheorem{problem}[theorem]{Problem}
\newtheorem{proposition}[theorem]{Proposition}

\newtheorem{lemma}[theorem]{Lemma}

\newtheorem{corollary}[theorem]{Corollary}

\theoremstyle{definition}

\newtheorem{example}[theorem]{Example}

\theoremstyle{remark}
\newtheorem{remark}[theorem]{Remark}

\numberwithin{equation}{section}

\begin{document}

\title[Cohomology with coefficients in twisted representations]{Cohomology of algebraic groups with coefficients in twisted representations}


\author{Antoine Touz\'e}
\address{Universit\'e Lille 1 - Sciences et Technologies\\
Laboratoire Painlev\'e\\
Cit\'e Scientifique - B\^atiment M2\\
F-59655 Villeneuve d'Ascq Cedex\\
France}
\curraddr{}
\email{antoine.touze@univ-lille.fr}
\thanks{This  work was supported in part by the Labex CEMPI (ANR-11-LABX-0007-01)}


\date{\today}

\begin{abstract}
This article is a survey on the cohomology of a reductive algebraic group with coefficients in twisted representations. A large part of the paper is devoted to the advances obtained by the theory of strict polynomial functors initiated by Friedlander and Suslin in the late nineties. The last section explains that the existence of certain `universal classes' used to prove cohomological finite generation is equivalent to some recent `untwisting theorems' in the theory of strict polynomial functors. We actually provide thereby a new proof of these theorems.
\end{abstract}

\maketitle

\bibliographystyle{amsplain}

\section{Introduction}

Let $G$ be an affine algebraic group (scheme) over a field $\kk$ of positive characteristic $p$. If $V$ is a rational representation of $G$, we may twist it by base change along the Frobenius morphism $\kk\to \kk$, $x\mapsto x^{p^r}$. The representation obtained is denoted by $V^{(r)}$.
This article deals with the problem of computing rational cohomology of $G$ with coefficients in such twisted representations. That is, we consider the following naive, basic question. 
\begin{center}
How different is $\HH^*(G,V^{(r)})$ from $\HH^*(G,V)$?
\end{center}
Cohomology with coefficients in twisted representations appears naturally in many problems regarding  rational cohomology of algebraic groups and the connections with other topics. Unfortunately, the naive question raised above has no easy answer. In recent years however, much effort has been made to find satisfactory partial answers. 

The first four sections of this article are a survey of the problem of computing cohomology with twisted coefficients. We describe some motivations and some important partial answers to this problem.  
As we do not assume that the reader is an expert of algebraic groups, we begin section \ref{sec-pm} by a short introduction to affine group schemes, their representations, and the associated cohomology theory. Then we present three motivations to the study of twisted representations, namely
(1) the link with the cohomology of finite groups of Lie type, (2) the study of simple representations of reductive algebraic groups, (3) cohomological finite generation theorems. This  list of motivations is not exhaustive. For example, we have left aside the connections with algebraic topology (via strict polynomial functors and unstable modules over the Steenrod algebra \cite{FFSS,Panorama}).

In sections \ref{sec-solnonfct} and \ref{sec-fct} we present classical and also more recent approaches and results regarding the cohomology of reductive group schemes with coefficients in twisted representations. In particular, section \ref{sec-fct} presents up-to-date results obtained by the functorial approach initiated by Friedlander and Suslin in \cite{FS}. Indeed, the theory of strict polynomial functors has recently brought very clean answers \cite{Chalupnik3,Tuan,TouzeUnivNew} to the naive question formulated above. In order that non-experts of strict polynomial functors can easily use these answers, we have formulated them purely in terms of (polynomial) representations of algebraic groups in section \ref{subsubsec-nofct}. Classical `group cohomologists' may  jump directly to this section to evaluate the interest of strict polynomial functors.

The last section is a bit more technical, but hopefully no less interesting. In this section we prove that the existence of some `universal cohomology classes' (which are one of the key tools in cohomological finite generation theorems \cite{TvdK}) is actually equivalent to the untwisting theorems of \cite{Chalupnik3,Tuan,TouzeUnivNew}. In particular, section \ref{subsec-sens-dur} gives a new uniform proof of all these untwisting theorems, which does not rely on the adjoint to the precomposition by the Frobenius twist. The latter is a technical tool introduced by M. Cha{\l}upnik \cite{Chalupnik3}, which was essential in the earlier proofs. Our proof also separates general arguments from statements specific to strict polynomial functors. We hope that this makes it readable for non-experts, and helps to understand where the `strict polynomial magic' really lies.

\subsubsection*{Acknowledgement} The author thanks the anonymous referee for very carefully reading a first version of the article and detecting several mistakes.

\section{Problem and motivations}\label{sec-pm}
\subsection{Group schemes and their representations}
We refer the reader to \cite{Waterhouse, Jantzen} for affine group schemes and their representations. A nice concise introduction to these topics is \cite{FriePan}. 
We recall here only a few basic definitions 	and conventions. 

Given a field $\kk$ we let $\kk-\mathrm{Alg}$ be the category of commutative finitely generated $\kk$-algebras with unit. An affine algebraic group scheme over $\kk$ is a representable functor $G$ from $\kk-\mathrm{Alg}$ to the category of groups. All group schemes in this article will be affine algebraic without further mention. A morphism of group schemes is a natural transformation $G\to G'$. Let $\kk[G]$ denote the finitely generated $\kk$-algebra representing a group scheme $G$, and let $\overline{\kk}$ be the algebraic closure of $\kk$. Then $G$ is said to be smooth, resp. connected, if  $\kk[G]\otimes_\kk \overline{\kk}$  is nilpotent-free, resp. has no nontrivial idempotent. 
For all $\kk$-vector spaces $V$, there is a functor 
$$GL_V: \kk-\mathrm{Alg}  \to  \mathrm{Groups}$$
sending a $\kk$-algebra $A$ to the group of invertible $A$-linear endomorphisms of $V\otimes_\kk A$. When $V$ has dimension $n$ this is a (smooth connected) group scheme sending $A$ to $GL_n(A)$, and we denote it by $GL_{n,\kk}$, or often by $GL_n$ when the ground field $\kk$ is clear from the context. A representation of a group scheme $G$ (or a $G$-module) is a $\kk$-vector space $V$ equipped with a natural transformation $\rho:G\to GL_V$. Such representations are also called rational representations, but we most often drop the adjective `rational'.

If $\kk$ is algebraically closed, smooth group schemes (such as $GL_n$, $SL_{n}$, $Sp_n$, $\mathrm{Spin}_n$\dots) identify with Zariski closed subgroups of some $GL_n(\kk)$. A finite dimensional rational representation of a smooth group scheme $G$ is then the same as a representation of $G$ such that the corresponding morphism $G\to GL(V)$ is a morphism of algebraic groups (i.e. a regular map). 

All the problems and results described in this paper remain hard and interesting over an algebraically closed field, the hypothesis that $\kk$ is algebraically closed does not bring any substantial simplification in the proofs. 

\subsection{Rational cohomology with coefficients in twisted representations}
We fix a field $\kk$ of positive characteristic $p>0$. The Frobenius morphism $\kk\to \kk$, $x\mapsto x^{p}$ and its iterates induce morphisms of group schemes (for $r\ge 0$)
$$\begin{array}[t]{cccc}
F^r:& GL_n & \to & GL_n\\
&[a_{ij}]& \mapsto & [a_{ij}^{p^r}]
\end{array}\;.$$
If $G$ is a subgroup scheme of $GL_n$ defined over the prime field $\mathbb{F}_p$, these morphisms restrict to $F^r:G\to G$. If $V$ is a vector space acted on by $G$ via $\rho:G\to GL_V$, we let $V^{(r)}$ denote the same vector space, acted on by $G$ via the composite $\rho\circ F^r$. The representation $V^{(r)}$ is called the \emph{$r$th Frobenius twist of $V$}, and we shall informally refer to such representations as \emph{twisted representations}.

Twisted representations play a prominent role in the representation theory of affine algebraic group schemes in positive characteristic. 
Given a group scheme $G$ and a representation $V$, we denote by $\HH^*(G,V)$ the extension groups $\Ext^*_G(\kk,V)$ between the trivial $G$-module $\kk$ and $V$, computed in the category of rational representations of $G$ (this category is abelian with enough injectives). 
The (rational) cohomology of $G$ with coefficients in twisted representations appears naturally in many situations, which motivates to study the following general problem.

\begin{mproblem}\label{mprob}
Try to understand or to compute the cohomology of $G$ with coefficients in twisted representations. In particular, if we know $H^*(G,V)$, what can we infer about $\HH^*(G,V^{(r)})$, for $r\ge 0$?
\end{mproblem}

In sections \ref{subsec-Stein}-\ref{subsec-Fin}, we give some concrete situations  where cohomology with coefficient in twisted representations naturally appear. Before this, we make some preliminary remarks regarding problem \ref{mprob}.

\begin{enumerate}
\item By definition, the representation $V^{(0)}$ equals $V$. Thus, problem \ref{mprob} is interesting for $r>0$. However, it is sometimes convenient to allow $r=0$ in order to obtain uniform statements, see e.g. the results in section \ref{subsubsec-nofct}.
\item One could ask similar questions about extensions groups between twisted representations, i.e. of the form $\Ext^*_G(V^{(r)},W^{(r)})$. This is not a more general question, since the representations $\mathrm{Hom}_\kk(V^{(r)},W^{(r)})$ and $\mathrm{Hom}_\kk(V,W)^{(r)}$ are isomorphic, so that there is a graded isomorphism:
$$\Ext^*_G(V^{(r)},W^{(r)})\simeq \HH^*(G, \mathrm{Hom}_\kk(V,W)^{(r)})\;.$$
\item Finite groups are group schemes, but problem \ref{mprob} is uninteresting for finite groups. Indeed, the Frobenius morphism $F^r:G\to G$ is then an isomorphism, thus $\HH^*(G,V^{(r)})$ is isomorphic to $\HH^*(G,V)$.
\item Problem \ref{mprob} may be asked for arbitrary affine algebraic group schemes, but in the sequel we concentrate mainly on \emph{reductive} group schemes. We refer the reader to \cite{Jantzen} for details about other kinds of group schemes.
\end{enumerate}

\begin{remark}[For experts] Our twist is denoted by `$^{[r]}$' in  \cite{Jantzen}, while `$^{(r)}$' refers there to a different way of twisting representations, defined by postcomposing $\rho$ by the action of the Frobenius morphism. However, if $V$ is defined over a subfield of $\mathbb{F}_{p^r}$ the two twists coincide: $V^{(r)}\simeq V^{[r]}$. As this is the case for most of the representations considered in this article, we have chosen to formulate all the motivations and results by using only the construction of twists defined by precomposition by the Frobenius morphism.
\end{remark}

\subsection{Cohomology of finite groups of Lie type}\label{subsec-CPSVdK}
We explain a first motivation to study problem \ref{mprob}, related to the cohomology of finite groups of Lie type.
Recall first some basic definitions. Let $\kk$ be a field with algebraic closure $\overline{\kk}$. A smooth connected group scheme $G$ is called \emph{reductive} if $G(\overline{\kk})$ does not contain any nontrivial normal connected unipotent subgroup (i.e. subgroup with unipotent elements only). It is \emph{split over $\kk$} if all its maximal connected diagonalizable subgroups are  isomorphic (over $\kk$) to a product of $n$ copies of the multiplicative group $\mathbb{G}_m$. The integer $n$ is called the rank of $G$. Typical $\kk$-split reductive group schemes of rank $n$ are the classical matrix groups $GL_n$, $SL_{n+1}$, $Sp_{2n}$ or $SO_{n,n}$ and $SO_{n,n+1}$ (which are split forms of $SO_{2n}$ and $SO_{2n+1}$, see \cite{SteinbergChev} for further details).

Let $p$ be a prime and consider a $\mathbb{F}_p$-split reductive group $G$. The functor $V\mapsto V^{(1)}$ is exact, so it induces for all $i$ a morphism
$$ H^i(G,V^{(r)})\to  H^i(G,V^{(r+1)})\;.$$
It is known that this map is injective \cite{CPS_inj} \cite[I.9.10]{Jantzen}, that is, \emph{twisting creates cohomology}. Moreover, Cline Parshall Scott and Van der Kallen discovered in \cite{CPSVdK} that this map is an isomorphism if $r$ is big enough, and that the stable value $\mathrm{colim}_r H^i(G,V^{(r)})$, which they call \emph{generic cohomology} and denote by $\HH^i_{\mathrm{gen}}(G,V)$, can be interpreted in terms of the cohomology of the finite groups of Lie type $G(\mathbb{F}_q)$.

To be more specific, for all finite extensions $\mathbb{F}_q$ of $\mathbb{F}_p$, any rational representation $V$ of $G$ yields a representation $V\otimes_{\mathbb{F}_p} \mathbb{F}_q$ of $G(\mathbb{F}_q)$. The resulting functor
$$\begin{array}{ccc} 
\{\text{Rational representations of $G$}\}&\to & \{\text{representations of  $G(\mathbb{F}_q)$}\}\\
V & \mapsto & V\otimes_{\mathbb{F}_p}\mathbb{F}_q
\end{array}$$ 
is exact, and sends $V^{(r)}$ and $V$ to the same $G(\mathbb{F}_q)$-modules provided $q|p^r$ (indeed the $r$-th iterated Frobenius morphism is then the identity of $\mathbb{F}_q$). Thus, if $q|p^r$, the $\mathbb{F}_q$ point functor induces a map
$$ H^i(G,V^{(r)})\otimes\mathbb{F}_q\to H^i(G(\mathbb{F}_q),V\otimes_{\mathbb{F}_p}\mathbb{F}_q) \;.$$
The main result in \cite{CPSVdK} asserts that this map is an isomorphism if $r$ and $q$ are big enough (with respect to $i$ and to explicit constants depending on $G$ and $V$).

\begin{example}[Quillen vanishing]\label{ex-Q-van}
Assume that $V=\kk$ is the trivial representation. Then $\kk^{(r)}=\kk$ for all $r$ so the colimit of the $H^i(G,\kk^{(r)})$ equals $H^i(G,\kk)$. Kempf vanishing theorem implies that this is zero in positive degrees. Hence, Cline Parshall Scott Van der Kallen's theorem says in this case that 
$ H^i(G(\mathbb{F}_q))=0$
for $q$ big enough with respect to $i$. This vanishing was first discovered by Quillen \cite{Quillen} in the case of the general linear group, and proofs for other classical groups can be found in \cite{FP}.
\end{example}

The trivial representations are the only ones such that $V=V^{(r)}$. Hence, to obtain  further concrete computations in the style of example \ref{ex-Q-van}, one needs to study our main problem \ref{mprob}.

\begin{remark}
The main theorem of \cite{CPSVdK} can also be formulated with more emphasis on finite groups, see section \ref{subsec-untwisting-finite}. 
\end{remark}

\subsection{Simple representations}\label{subsec-Stein}
We now describe a second motivation for problem \ref{mprob}, related to the representation theory of reductive group schemes. Let $\kk$ be a field, and let $G$ be a $\kk$-split reductive group scheme. 

Chevalley classified \cite{Chevalley} the simple representations of $G$, using the root system associated to $G$ (the procedure used to classify simple representations is called \emph{highest weight theory}). 
However, our knowledge of simple representations is far from being satisfactory in positive characteristic. For example, for classical matrix groups, we don't know the dimensions of the simple representations in general, nor can we compute the Ext-quiver of $G$.

Twisted representations naturally appear when studying simple representations of $G$.  For example, if $V$ is a nontrivial simple representation of $V$, then 
$$V, V^{(1)}, V^{(2)}, \dots, V^{(r)},\dots $$ 
is an infinite family of pairwise non isomorphic simple representations. More generally, Steinberg tensor product theorem \cite{Steinberg} \cite[II.3.17]{Jantzen} says that all simple representations are of the form 
$$V_0\otimes V_1^{(1)}\otimes \dots \otimes V_r^{(r)}$$
for some $r\ge 0$ and some simple representations $V_i$ belonging to a small set of simple representations, namely those with $p$-restricted highest weight.
\begin{example}
Take $G=SL_{n+1}$. The simple representations of $SL_{n+1}$ are in bijection with the set $\Lambda_+$ of partitions of integers in at most $n$ parts (i.e. the set of $n$-tuples $(\lambda_1,\dots,\lambda_n)$ of nonnegative integers satisfying $\lambda_1\ge \dots\ge \lambda_n$). The simple representations with $p$-restricted highest weight correspond through this bijection with the finite subset $\Lambda_1\subset \Lambda_+$ of partitions satisfying $\lambda_i-\lambda_{i+1}<p$ for $0\le i<n$ and $\lambda_n<p$. 
\end{example}

The structure of simple representations of $G$ given by the Steinberg tensor product theorem implies that any general problem involving the cohomology of simple representations of $G$ requires to study our main problem \ref{mprob}.

\subsection{Cohomological finite generation}\label{subsec-Fin}

Let $G$ be a finite group, and $\kk$ a field of positive characteristic $p$. The cohomology of $G$ with trivial coefficients $\HH^*(G,\kk)$ is then a graded $\kk$-algebra, and Evens proved \cite{Evens} that it is finitely generated. Finite generation is an important input to build a theory of support varieties, which establishes a bridge between geometry (subvarieties of $\mathrm{Spec}\, \HH^{ev}(G,\kk)$) and algebra (representations of $G$).
The development of support varieties was a motivation to understand the cohomology of the wider class of \emph{finite group schemes}. Finite group schemes are the group schemes represented by finite dimensional $\kk$-algebras. They include finite groups, enveloping algebras of restricted Lie algebras, and kernels of iterated Frobenius morphisms $F^r:G\to G$. Friedlander and Suslin proved the following theorem. 
\begin{theorem}[{\cite[Thm 1.1]{FS}}]\label{thm-FS}
Let $G$ be a finite group scheme over a field $\kk$ and $M$ a finite dimensional $G$-module. Then $\HH^*(G,\kk)$ is a finitely generated algebra, and $\HH^*(G,M)$ is a finite module over it.
\end{theorem}

Cohomology with coefficients in twisted representations play an important role in 
Friedlander and Suslin's proof. Let us briefly explain this in more detail. The proof follows the same basic principle as Evens' proof for finite groups. They consider a spectral sequence of graded algebras, with noetherian initial page  and converging to the desired finite group scheme cohomology $\HH^*(G,\kk)$. Then, each page of the spectral sequence is finitely generated, and to prove finite generation of $\HH^*(G,\kk)$, one needs to prove that the spectral sequence stabilizes after a finite number of pages. This is achieved by constructing sufficiently many permanent cocycles in the spectral sequence. In Evens' proof, the permanent cocycles are constructed with Even's norm map. For finite group schemes, there is no such map available. Instead, the permanent cocycles are constructed from cohomology classes living in $\HH^*(GL_n,\mathfrak{gl}_n^{(r)})$. 
The computation of this cohomology with twisted coefficients (see example \ref{ex-FScomputation} below) is a breakthrough of \cite{FS}, and solves an important case of problem \ref{mprob}.

Actually, Evens (and later Friedlander and Suslin, although this is not explicitly stated in \cite{FS}) proved a more general result. If $A$ is a commutative $\kk$-algebra acted on by $G$ by algebra automorphisms, then $\HH^*(G,A)$ is a graded commutative algebra. They prove that $\HH^*(G,A)$ is finitely generated as soon as $A$ is so. The most general version of this result was proved in \cite{TvdK}. Let $\overline{\kk}$ be the algebraic closure of $\kk$ and let $G$ be a group scheme such that the identity component of the algebraic group $G(\overline{\kk})$ is reductive. Thus $G$ may be reductive, finite, or an extension of such groups schemes. It is known from invariant theory that these group schemes are exactly the ones with the finite generation (FG) property, that is the invariant $\kk$-algebra $A^G=\HH^0(G,A)$ is finitely generated as soon as $A$ is.
\begin{theorem}[{\cite[Thm 1.1]{TvdK}}]
Let $G$ be a group scheme over $\kk$, satisfying (FG) property. Let $A$ be a finitely generated commutative algebra, acted on by $G$ by automorphisms of algebras. Then $H^*(G,A)$ is finitely generated.
\end{theorem}

As in Friedlander and Suslin's theorem, cohomology with coefficients in twisted representations plays an important role in the proof. In particular the proof requires the construction of some cohomology classes in $\HH^*(GL_n,\Gamma^{d}(\mathfrak{gl}_n)^{(1)}\,)$, $d\ge 0$, in order to produce permanent cycles in spectral sequences. Here $\Gamma^d(\mathfrak{gl}_n)$ denotes the $d$-th divided power of $\mathfrak{gl}_n$, i.e.  $\Gamma^d(\mathfrak{gl}_n)=(\mathfrak{gl}_n^{\otimes d})^{\Si_d}$ is the subrepresentation of elements in $\mathfrak{gl}_n^{\otimes d}$ which are invariant under the action of the symmetric group (acting by permuting the factors of the tensor product). The problem of constructing non-zero classes in $\HH^*(GL_n,\Gamma^{d}(\mathfrak{gl}_n)^{(1)}\,)$ is clearly an instance of problem \ref{mprob}, but we will explain in section \ref{sec-eq} that the two problems are actually equivalent.

\section{Some results regarding the cohomology of twisted representations}\label{sec-solnonfct}
In this section $\kk$ is a field of positive characteristic $p$ and $G$ is a reductive algebraic group over $\kk$, which is assumed to be split and defined over $\mathbb{F}_p$, for example one of the classical group schemes $GL_n$, $SL_{n+1}$, $SO_{n,n}$, $SO_{n,n+1}$, $Sp_{2n}$. We describe two approaches to the cohomology of $G$ with coefficients in twisted representations. The first one is classical and described e.g. in the reference book \cite{Jantzen}, while the second one is more recent \cite{CPSt}. The content of this section is not needed in the remainder of the article. We have included this section in order to give an overview of the techniques used to attack problem \ref{mprob} without using strict polynomial functor technology. The reader can compare these important results to the results based on strict polynomial functor technology, described in section \ref{sec-fct}.

\subsection{Untwisting by using Frobenius kernels}\label{subsec-untwist-Frob}
Let $G_r$ denote the $r$-th Frobenius kernel of $G$, that is the scheme-theoretic kernel of the $r$-th iterated Frobenius $F^r:G\to G$. 
Then $G_r$ is a normal subgroup scheme of $G$ and for all $G$-module $U$, the Lyndon-Hochschild-Serre spectral sequence of the extension $G_r \to G \twoheadrightarrow G/G_r$  has the form
$$E_2^{p,q}(U)=\HH^p(G/G_r, \HH^q(G_r,U))\;\Rightarrow\; \HH^{p+q}(G,U)\;.$$

If $U=V^{(r)}$, this spectral sequence can be used to `untwist' the action, and to obtain some information on $\HH^*(G,V^{(r)})$ from cohomology with untwisted coefficients. Indeed, the quotient $G/G_r$ fits into a commutative diagram
$$\xymatrix{
G\ar[r]^-{F^r}\ar@{->>}[rd]^-{\pi}& G\ar[d]^-{\simeq}\\
& G/G_r}\;.$$  
In particular, for all $G$-modules $W$ the following three statements are equivalent:
\begin{enumerate}
\item[(i)] $G_r$ acts trivially on $W$
\item[(ii)] There is a $G/G_r$-module $V$ such that $\rho_W=\rho_V\circ \pi$.  
\item[(iii)] There is a $G$ module $W^{(-r)}$ such that $W=(W^{(-r)})^{(r)}$.
\end{enumerate}
Thus the LHS spectral sequence can be rewritten as
$$E_2^{p,q}(V^{(r)})=\HH^p(G, \HH^q(G_r,\kk)^{(-r)}\otimes V)\;\Rightarrow\; \HH^{p+q}(G,V^{(r)})\;.\qquad (*)$$

The cohomology of $\HH^*(G_r,\kk)$ is not known in general, and the differentials of the spectral sequence may very well be nonzero. However, one can overcome these difficulties when $r=1$ and $p$ is big enough (with some explicit bound, e.g. for $G=SL_n$ or $GL_n$ then $p>n$ is big enough). Under these hypotheses Friedlander and Parshall \cite{FPar1,FPar2} and Andersen and Jantzen \cite{AJ} determined $\HH^*(G_1,\kk)^{(-1)}$. As a very nice feature, this $G$-module has a `good filtration', in particular it has no higher $G$-cohomology. If in addition $V$ also has a good filtration then the  second page of the LHS spectral sequence satisfies $E^{p,q}(V^{(r)})=0$ for all positive $p$, hence the spectral sequence collapses and one obtains the following result.
\begin{theorem}\label{thm-untwistFrob}
Assume that $G$ is a reductive group and that $p>h$ ($h$ denotes the Coxeter number of $G$, if $G=GL_n$ or $SL_n$, $h=n$). If $V$ has a good filtration, there is an isomorphism:
$$ \HH^0(G,\HH^q(G_1,\kk)^{(-1)}\otimes V)\simeq \HH^{q}(G,V^{(1)})\;.$$
\end{theorem}

The approach to the computation of cohomology with twisted coefficients described above can be adapted to slightly more general $G$-modules as coefficients, as in \cite[Prop 4.8]{AJ}. Also the dimensions of the $\HH^{q}(G,V^{(1)})$ can be computed by explicit combinatorial formulas. As it is done in \cite{AJ}, it is also possible to study some special cases when $r=2$ or for $p= h$ or $p=h-1$. However, all this requires nontrivial technical work (analysis of weights or differentials in the spectral sequence) and cannot be pursued very far to obtain complete computations of $\HH^*(G,V^{(r)})$.
 
\begin{remark}\label{rk-autres-util}
Although the technique of untwisting by using Frobenius kernels reaches quickly its limits for complete computations, it can still be very useful to obtain qualitative information on twisted representations. For example, it plays an important role in the proof of cohomological finite generation for reductive algebraic groups, see \cite{vdKGrosshans,TvdK}. As another example, by carefully analysing the weights of the $G$-modules appearing at the second page of the LHS spectral sequence, it is sometimes possible to detect vanishing zones in $E_2^{*,*}$, therefore leading to vanishing results for $H^j(G,V^{(r)})$ as in \cite[Thm 5.2]{CPSt}. See theorem \ref{thm-vanish} below.
\end{remark}

\subsection{Untwisting by using finite groups of Lie type}\label{subsec-untwisting-finite}
Let us come back to Cline Parshall Scott van der Kallen's comparison theorem \cite{CPSVdK}. In section \ref{subsec-CPSVdK}, we presented this theorem as an interpretation of $\mathrm{colim}_r\HH(G,V^{(r)})$. The reader primarily interested in finite groups might prefer the following alternative presentation. Take $\kk=\overline{\mathbb{F}_p}$ and let $G=G(\kk)$ be a reductive algebraic group defined and split over $\mathbb{F}_p$ (e.g. $G=SL_{n+1}$). Let $V$ be a rational representation of $G(\kk)$. All the finite subgroups $G(\mathbb{F}_q)$ act on $V$, and the generic cohomology of $V$ is 
$$\HH^i_{\mathrm{gen}}(G,V)=\lim_q \HH^i(G(\mathbb{F}_q),V)\;.$$ 
The main theorem of \cite{CPSVdK} asserts that $\HH^i_{\mathrm{gen}}(G,V)=\HH^i(G(\mathbb{F}_q),V)$ for $q$ big enough, and moreover this is equal to the rational cohomology $\HH^i(G,V^{(r)})$ for some big enough $r$. 

Now assume that we are interested in finite group cohomology. We have a finite group $G(\mathbb{F}_q)$ and a rational $G$-module $V$, and we wish to compute $H^*(G(\mathbb{F}_q),V)$ by using CPSvdK comparison theorem and the rational cohomology of $G$. Then we face two practical problems:
\begin{enumerate}
\item We need to understand $\mathrm{colim}_r\HH^*(G,V^{(r)})$ rather than $\HH^*(G,V)$ (this is our main problem \ref{mprob}).
\item Maybe our field $\mathbb{F}_q$ is not big enough in order that $\mathrm{colim}_r\HH^*(G,V^{(r)})$ is isomorphic to $\HH^*(G(\mathbb{F}_q),V)$.
\end{enumerate}
A recent article of Parshall, Scott and Stewart \cite{CPSt} solves both problems at the same time. Their result applies when $V$ is a simple $G$-module or more generally when $V=\mathrm{Hom}_\kk(U,U')$ for simple $G$-modules $U$, $U'$. For the sake of simplicity, we explain their method when $V$ is simple. The results needs further mild restrictions on the reductive group $G$, namely that $G$ is simply connected and semisimple (e.g. $G=SL_{n+1}$), so we assume $G$ satisfies these restrictions in this section.  

The approach of \cite{CPSt} can be decomposed in two steps.
The first step uses the untwisting technique of section \ref{subsec-untwist-Frob} as explained in remark \ref{rk-autres-util} to prove a cohomological vanishing result. To state this result, recall from section \ref{subsec-Stein} that $V$ decomposes as a tensor product 
$$V=V_0\otimes V_1^{(1)}\otimes \dots\otimes V_{k}^{(k)}\;.$$
where the $V_i$ are simple modules with $p$-restricted highest weight. Since $G$ is semi-simple the decomposition is unique. We call it the \emph{Steinberg decomposition of $V$} in the sequel.

\begin{theorem}[{\cite[Thm 5.2]{CPSt}}]\label{thm-vanish}
For all $i$ there exists an (explicit) integer $d=d(G,i)$ such that for all $V$ with more than $d$ nontrivial simple factors $V_i$ in their Steinberg decomposition, and all $j\le i$ one has:
$$\HH^j(G,V^{(r)})=0=\HH^j(G,V)\;.$$ 
\end{theorem} 
Using induction from $G(\mathbb{F}_q)$ to $G$ and relying on the filtration of $\mathrm{Ind}_{G(\mathbb{F}_q)}^G(\kk)$ provided in \cite{BNP}, one can then deduce a similar vanishing result \cite[Thm 5.4]{CPSt} for $\HH^j(G(\mathbb{F}_q),V)$, which is valid  even when $q$ is not big enough to apply CPSvdK comparison theorem. 

The second step studies the remaining cases, i.e. when $V$ is a tensor product of a small (i.e. $\le d$) number of nontrivial twisted simple representations of $G$. In these remaining cases, comparison with $G(\mathbb{F}_q)$ cohomology can be used to untwist the coefficients in the following way. 
First, for all $e\ge 0$ the Frobenius maps $F^e:G\to G$ induce isomorphisms of the finite groups $G(\mathbb{F}_q)$ hence $\HH^*(G(\mathbb{F}_q),V^{(e)})$ does not depend (up to isomorphism) on the twisting $e$. 
Second, all the $V^{(e)}$ are simple $G(\mathbb{F}_q)$-modules, but they need not be isomorphic in general. Of course if $q=p^u$, $F^u$ is the identity map on $G(\mathbb{F}_q)$, hence $V^{(u)}=V$. Thus Frobenius twisting induces a $\mathbb{Z}/u\mathbb{Z}$-action on the set of simple representations of $G(\mathbb{F}_q)$. 
\begin{example}\label{ex-cycl-action}
Take $u=5$ and $V=V_0\otimes V_2^{(2)}$, with $V_0$ and $V_2$ non-isomorphic. Then the orbit of $V$ under the action of $\mathbb{Z}/5\mathbb{Z}$ is:
$$ V_0\otimes V_2^{(2)}\longrightarrow V_0^{(1)}\otimes V_2^{(3)} \longrightarrow V_0^{(2)}\otimes V_2^{(4)} \longrightarrow V_2\otimes V_0^{(3)} \longrightarrow V_2^{(1)}\otimes V_0^{(4)}\;. $$
\end{example}
Moreover, the simple $G(\mathbb{F}_q)$-representations in the orbit of $V$ under the action of $\mathbb{Z}/u\mathbb{Z}$ can be interpreted as restrictions of the simple $G$-modules $V^{(e)}$, but one can also interpret each of them in a unique way as the restriction of a simple $G$-module whose tensor product decomposition $(**)$ involves only $k$-th twisted $G$-modules for $k< u$. 
These simple $G$-modules are called the \emph{$q$-shifts of $V$} in \cite{CPSt}.
 \begin{example}
In example \ref{ex-cycl-action}, the $q$-shifts of $V$ are: $V$, $V^{(1)}$, $V^{(2)}$, $W$ and $W^{(1)}$, where $W=V_2\otimes V_0^{(3)}$. Note that \emph{as a $G$-module}, $W$ is not isomorphic to a Frobenius twist of $V$.
\end{example}

Now since $V$ has a small number $d$ of nontrivial factors, it is possible to find a (computable) bound $u_0=u_0(G,i)$ big enough with respect to $d$, such that for all $u\ge u_0$ and $q=p^u$, and for all $V$ with $p^u$-restricted highest weight, there is a $q$-shift $V'$ whose tensor product decomposition $(**)$ is such that: (i) it starts with a long enough chain of trivial simple representations and (ii) it ends with a long enough chain of trivial representations. Condition (i) ensures that $V'=W^{(s)}$ is already twisted enough to apply CPSvdK theorem without further twisting, while condition (ii) ensures that $\mathbb{F}_q$ is big enough to apply CPSvdK theorem. This, together with theorem \ref{thm-vanish} imply the following theorem.
\begin{theorem}[{\cite[Thm 5.8]{CPSt}}]\label{thm-untwistFq}
Let $G$ be a semi-simple group scheme, simply connected, split and defined over $\mathbb{F}_p$. Then there exists a nonnegative integer $u_0(G,i)$ such that for all $u\ge u_0$ and $q=p^u$, the following holds. 
For all simple $G$-modules $V$ with $p^u$-restricted highest weight (i.e. involving only $k$-th Frobenius twists for $k<u$ in their Steinberg decomposition), one can find a $q$-shift $V'$ such that for all $j\le i$, there is a chain of isomorphisms:
$$\HH^j(G(\mathbb{F}_q),V)\simeq \HH^j(G(\mathbb{F}_q),V')\simeq \HH_{\mathrm{gen}}^j(G,V')\simeq \HH^j(G,V')\;.$$
\end{theorem}

This theorem removes the necessity of understanding Frobenius twists to compute finite group cohomology from rational cohomology of $G$, as well as problems with the size of the field. 
If one is interested in rational cohomology of group schemes rather than in the cohomology of finite groups, Then by taking $r$ and $q'$ big enough (both depending on $V$ and with $q'\ge q$ with $q$ as in theorem \ref{thm-untwistFq}), one obtains a chain of isomorphisms for all $j\le i$:
\begin{align*}\HH^j(G,V^{(r)})\simeq \HH_{\mathrm{gen}}^j(G,V)&\simeq \HH^j(G(\mathbb{F}_{q'}),V)\\&\simeq \HH^j(G(\mathbb{F}_{q'}),V')\simeq \HH_{\mathrm{gen}}^j(G,V')\simeq \HH^j(G,V')\;.
\end{align*}
One may draw two opposite conclusions from such a chain of isomorphisms. On the one hand, one may think that it is possible to avoid studying further our main problem \ref{mprob} when coefficients are simple representations. On the other hand, this can be seen as an additional motivation to study problem \ref{mprob}. Indeed, $\HH^*(G,V')$ is mysterious in general, and understanding some general properties of high Frobenius twists of $V$ might bring some interesting new information on $V'$. 

\section{A functorial approach to the cohomology with twisted coefficients}\label{sec-fct}

\subsection{Polynomial representations of $GL_n$ and strict polynomial functors}\label{subsubsec-pol}
We now give a quick introduction to polynomial representations and strict polynomial functors. For the sake of simplicity, we make the assumption (in this section \ref{subsubsec-pol} only)  that the ground field $\kk$ is infinite of arbitrary characteristic. We refer the reader to \cite[Sections 2 and 3]{FS}, \cite{TouzeClass} or \cite{Krause} for a presentation which is valid over an arbitrary field (or even more generally over an arbitrary commutative ring). 

Since $\kk$ is infinite, finite dimensional rational representations of $GL_n$ identify with group morphisms $\rho:GL_n(\kk)\to GL(V)\simeq GL_m(\kk)$, $[a_{i,j}]\mapsto [\rho_{k,\ell}(a_{i,j})]$ whose coordinates functions may be written in the form 
$$\rho_{k,\ell}(a_{i,j})=\frac{P_{k,\ell}(a_{i,j})}{\det([a_{i,j}])^{\alpha_{k,\ell}}}$$
where $P_{k,\ell}(a_{i,j})$ is a polynomial in the $n^2$-variables $a_{i,j}$ and $\alpha_{k,\ell}$ is a nonnegative integer. Thus the coordinate functions are \emph{rational functions} of the matrix coordinates $a_{i,j}$. The rational representation 
$(V,\rho)$ is called polynomial (resp. of degree $d$) if all the $\rho_{k,\ell}$ are polynomial functions (resp. of degree $d$). Finally, infinite dimensional polynomial representations of degree $d$ are the rational representations which can be written as a union of finite dimensional subrepresentations which are polynomial of degree $d$.

\begin{example}\label{ex-NSGT}
The defining representation $V=\kk^n$ of $GL_n$ is polynomial (of degree $1$), its $d$-th tensor power $V^{\otimes d}$ is polynomial (of degree $d$) its $d$-th symmetric power  $S^d(V)=(V^{\otimes d})_{\Si_d}$ is polynomial (of degree $d$) and its $d$-th divided power $\Gamma^d(V)=(V^{\otimes d})^{\Si_d}$ is polynomial (of degree $d$) as well.
\end{example}

Strict polynomial functors can be thought of as a natural way to generalize example \ref{ex-NSGT}. To be more specific, let $\mathcal{V}_\kk$ be the category of vector spaces over $\kk$ and $\mathcal{V}^f_\kk$ its subcategory of finite dimensional vector spaces. A functor $F:\mathcal{V}^f_\kk\to \mathcal{V}_\kk$ is called strict polynomial of degree $d$ if for all vector spaces $U$, $V$, the coordinate functions of the map 
$$\begin{array}{cccc}
F_{U,V}:&\mathrm{Hom}_\kk(U,V) &\to &\mathrm{Hom}_\kk(F(U),F(V)) \\
&f & \mapsto & F(f)
\end{array}$$
are polynomial functions of degree $d$. 
\begin{remark}
Strict polynomial functors are simply called `polynomial functors' in \cite{MacDonald}, and it is not clear from the definition why the word `strict' should be used. The reason is that the notion of a polynomial functor was already defined by Eilenberg and Mac Lane much before \cite{EML}. Strict polynomial functors are polynomial in the sense of Eilenberg-Mac Lane, but the converse is not true, which justifies the word `strict'. 
\end{remark}

For $V=\kk^n$, the map $F_{V,V}$ restricts to a polynomial action $\rho:GL_n(\kk)\to GL_\kk(F(V))$ on $F(V)$. The polynomial representations of example \ref{ex-NSGT} are obtained in this way from the degree $d$ strict polynomial functors $\otimes^d:V\mapsto V^{\otimes d}$, $S^d:V\mapsto S^d(V)$ and $\Gamma^d:V\mapsto \Gamma^d(V)$. Let us denote by $\mathcal{P}_\kk$ the category of strict polynomial functors of finite degree and natural transformations. Evaluation on $V=\kk^n$ yields an exact functor
$$\mathrm{ev}_{\kk^n}\;:\;\PP_\kk\to \text{Rational representations of $GL_n$} \;.$$
The image of $\mathrm{ev}_{\kk^n}$ consists exactly of the polynomial representations. Friedlander and Sulin proved \cite[Cor 3.13]{FS} that the graded map induced by evaluation:
$$ \Ext^*_{\PP_\kk}(F,F')\to  \Ext^*_{GL_n}(F(\kk^n),F'(\kk^n))$$
is an isomorphism provided $n$ is greater or equal to $\deg F$ and $\deg F'$. As shown in \cite{FS}, it is often easier in practice to compute extensions between polynomial representations by computing them inside $\PP_\kk$ than by computing them in the category of rational $GL_n$-modules. For example, the following fundamental computation is out of reach of the untwisting techniques described in sections \ref{subsec-untwist-Frob} and \ref{subsec-untwisting-finite}. 
\begin{example}\label{ex-FScomputation}
Let $\kk$ be a field of positive characteristic $p$, and let $E_r$ denote the graded truncated polynomial algebra:
$$E_r=S^*(e_1,\dots,e_r)\big/(e_1^p=\dots = e_r^p=0)$$
where the generators $e_i$ are homogeneous of degree $2p^{i-1}$. In particular, as a graded vector space one has:
$$E_r^i =\begin{cases} \kk & \text{if $i$ is even and $0\le i < 2p^r$}\\
0 & \text{otherwise}\end{cases}\;.$$
By using strict polynomial functors, Friedlander and Suslin showed in \cite[Thm 4.10]{FS} that for all $r$ and all $n\ge p^r$, there is an isomorphism of graded algebras:
$$\HH^*(GL_n,\mathfrak{gl}_n^{(r)})=\Ext^*_{GL_n}(\kk^{n\,(r)},\kk^{n\,(r)})\simeq E_r \;.$$
\end{example}

\begin{remark}[Strict polynomial functors versus modules over Schur algebras]
Let $\PP_{d,\kk}$ be the full subcategory of $\PP_{\kk}$ whose objects are the strict polynomial functors which are homogeneous of degree $d$ . Then $\PP_{d,\kk}$ is isomorphic to the category of modules over the Schur algebra $S(n,d)=\mathrm{End}_{\kk\Si_d}((\kk^n)^{\otimes d})$ if $n\ge d$. Modules over the Schur algebra are a classical subject of representation theory, which goes back to the work of Schur \cite{Schur} in characteristic zero, and which is studied in Green's book \cite{Green} in positive characteristic. Thus one may prefer to do $\Ext$-computations in the classical category of $S(n,d)$-Modules rather than in the more unusual category $\PP_{d,\kk}$ of strict polynomial functors. This is not a good idea, for many computations in $\PP_\kk$ ultimately rely on composing functors. Such a composition operation makes sense with strict polynomial functors, but is harder to define and manipulate in categories of modules over Schur algebras.
\end{remark}

\subsection{Frobenius twists}
Let $\kk$ be a field of positive characteristic $p$. The $r$-th Frobenius twist functor $I^{(r)}$ is a strict polynomial functor of homogeneous degree $p^r$, which may be described as the subfunctor of $S^{p^r}$ such that for each finite dimensional vector space $V$, $I^{(r)}(V)$ is the subspace of $S^{p^r}(V)$ spanned by all $p^r$-th powers $v^{p^r}$, with $v\in V$. 
Precomposition by $I^{(r)}$ is the functorial version of taking the $r$-th Frobenius twist of a representation of $GL_n$. Indeed, if $F$ is a strict polynomial functor of degree $d$, the composition $F\circ I^{(r)}$ is a strict polynomial functor of degree $dp^r$, and evaluating on $\kk^n$ yields an isomorphism of $GL_n$-modules:
$$\mathrm{ev}_{\kk^n} (F\circ I^{(r)})\simeq  (\mathrm{ev}_{\kk^n} F)^{(r)}\;.$$
Composite functors of the form $F\circ I^{(r)}$ are often denoted suggestively by $F^{(r)}$ and we shall informally refer to such functors as `twisted functors'. Problem \ref{mprob} translates as follows in the realm of strict polynomial functors.
\begin{problem}\label{mprob-bis}
Try to understand or to compute the $\Ext$ between twisted functors. In particular, if we know $\Ext^*_{\PP_\kk}(F,G)$, what can we infer about $\Ext^*_{\PP_\kk}(F^{(r)},G^{(r)})$, for $r>0$?
\end{problem}
One of the recent important achievements of the theory of strict polynomial functors is a complete solution to problem \ref{mprob-bis}. To state the solution we need an auxiliary operation on strict polynomial functors, which was introduced in \cite{TouzeENS}. Given a strict polynomial functor $F$ and a finite dimensional vector space $W$, we let $F_W$ denote strict polynomial functor such that $F_W(V)=F(W\otimes V)$. Now if $W$ is graded, then $F_W$ canonically inherits a grading. We call this construction \emph{parametrization} (indeed, the new functor $F_W$ is just the old one with a parameter $W$ inserted). 
One may think of the grading on parametrized functors as follows. Let $\PP_{\kk}^*$ denote the category of graded strict polynomial functors, and morphisms of strict polynomial functors preserving the grading. Parametrization by $W$ yields an \emph{exact} functor
$$\begin{array}{ccc}
\PP_\kk & \to & \PP_\kk^*\\
F & \mapsto & F_W
\end{array}\;.$$
Thus, to understand the grading on parametrized functors, it suffices to understand the grading on parametrized injectives. But injectives in $\PP_\kk$ are (direct summands of products of) symmetric tensors of the form $V\mapsto S^{d_1}(V)\otimes\dots \otimes S^{d_k}(V)$. On symmetric tensors, the grading induced by parametrization is simply the usual grading that one obtains by considering $W\otimes V$ as a graded vector space (with $V$ concentrated in degree zero) and applying the functor $S^{d_1}\otimes\dots \otimes S^{d_k}$.
The following result was conjectured in \cite{TouzeENS}, where it was verified for several families of pairs of functors $(F,G)$. The result generalizes many previously known computations \cite{FS,FFSS,Chalupnik1}. One may find two different proofs of it, namely in \cite{TouzeUnivNew} and in \cite{Chalupnik3}. Both proofs rely on an idea of M. Cha{\l}upnik, namely using the (derived) adjoint of precomposition by the Frobenius twist. (But see section \ref{subsec-sens-dur}).
\begin{theorem}\label{thm-Tformula}
For all strict polynomial functors $F,G$, there is a graded isomorphism:
$$\Ext^*_{\PP_\kk}(F^{(r)},G^{(r)})\simeq \Ext^*_{\PP_\kk}(F,G_{E_r})\;,$$
where the degree on the right hand side is understood as the total degree obtained by adding the $\Ext$-degree with the degree of the graded functor $G_{E_r}$ obtained by parametrizing $G$ by the graded vector space $E_r=\Ext^*_{\PP_\kk}(I^{(r)},I^{(r)})$ explicitly described in example \ref{ex-FScomputation}.
\end{theorem}

Theorem \ref{thm-Tformula} can be efficiently used in practical computations. For example, the computations of \cite{Chalupnik1} and most computations of \cite{FFSS} are instances of the following example.

\begin{example}
Assume that for all vector spaces $W$, $\Ext^{>0}_{\PP_\kk}(F,G_W)=0$. Such a condition is satisfied in the following concrete cases: if $F$ is projective, or if $G$ is injective, or if the $GL_d$-module $F(\kk^d)$ has a standard filtration and the $GL_d$-module $G(\kk^d)$ has a costandard filtration\footnote{Indeed, that the $GL_d$-module $G(\kk^d)$ has a costandard filtration is equivalent to the fact that the functor $G$ has a Schur filtration, i.e. a filtration whose associated graded is a direct sum of Schur functors as defined in \cite{ABW}. It then follows from \cite[Thm II.2.16]{ABW} that the parametrized functor $G_W$ also has a Schur filtration. The Ext condition follows by a highest weight category argument.} where $d=\max\{\deg F,\deg G\}$. Then the $\Ext$ computation between $F^{(r)}$ and $G^{(r)}$ can be recovered as a $\mathrm{Hom}$ computation via the graded isomorphism:
$$\Ext^*_{\PP_\kk}(F^{(r)},G^{(r)})\simeq \mathrm{Hom}_{\PP_\kk}(F,G_{E_r})\;.$$
Moreover, in the concrete cases given above, the latter $\mathrm{Hom}$ computation is easy to perform.
\end{example}

In general, it may not be easy to compute $\Ext^*_{\PP_\kk}(F,G_{E_r})$. However, even when these $\Ext$-groups do not seem easy to compute, one can draw many interesting qualitative results from the isomorphism of theorem \ref{thm-Tformula}. For example, the graded vector space $\mathrm{Hom}_{\PP_\kk}(F,G_{E_r})$ is a direct summand of $\Ext^*_{\PP_\kk}(F^{(r)},G^{(r)})$, which provides lots of nontrivial cohomology classes. 
Here is another way to use theorem \ref{thm-Tformula}. Recall that $E_r=\bigoplus_{i=0}^{p^r-1}\kk[2i]$, where $\kk[s]$ denotes a copy of $\kk$ placed in degree $s$. In particular, the graded functor $\bigoplus_{i=0}^{p^r-1} G_{\kk[2i]}$ is a direct summand of $G_{E_r}$. If $G$ is a homogeneous functor of degree $d$, then $G_{\kk[2i]}=G[2di]$ (a copy of $G$ placed in degree $2di$). Thus theorem \ref{thm-Tformula} has the following consequence.
\begin{corollary}\label{cor-per}
If $G$ is homogeneous of degree $d$ then 
$$\Ext^*_{\PP_\kk}(F^{(r)},G^{(r)})\simeq \bigoplus_{i=0}^{p^r-1}\Ext^{*+2di}_{\PP_\kk}(F,G) \;\oplus \; \text{Another graded summand}\;.$$
\end{corollary}
In other words, $\Ext^*_{\PP_\kk}(F^{(r)},G^{(r)})$ contains $p^r$ shifted copies of $\Ext^{*}_{\PP_\kk}(F,G)$ as direct summands. This explains periodicity phenomena which were often observed empirically in computations. One can refine this idea at the price of using a bit of combinatorics. If $G$ is a homogeneous functor of degree $d$, then the functor $G(V_1\oplus \dots\oplus V_N)$ with $N$ variables can be decomposed in a direct sum of homogeneous strict polynomial functors of $N$ variables. Applying this to $E_r\otimes V= V[0]\oplus V[2]\oplus\dots \oplus V[2p^r-2]$, we obtain a decomposition of $G_{E_r}$. To state this decomposition, we denote by $\Lambda(d,k)$ the set of compositions of $d$ into k parts (i.e. of tuples $\mu=(\mu_1,\dots,\mu_{k})$ of nonnegative integers such that $\mu_1+\dots+\mu_{k}=d$) and by $\Lambda_+(d,k)$ the subset of partitions into $k$ parts (i.e. those compositions satisfying $\mu_0\ge \dots\ge \mu_k$). By reordering a composition $\mu$, one obtains a partition which we denote by $\pi(\mu)$. Then there are strict polynomial functors $G_\lambda$ indexed by partitions of $d$ and an isomorphism of graded functors (the integer in brackets indicates in which degree the copy of $G_\lambda$ is placed)
$$G_{E_r}= \bigoplus_{\mu\in \Lambda(d,p^r)} G_{\pi(\mu)} \left[\sum_{i=1}^{p^r} 2(i-1)\mu_i\right]\;. $$
For all partitions $\lambda$ we let $E_\lambda=\Ext^*_{\PP_\kk}(F,G_\lambda)$. This is a vector space of finite total dimension provided $F$ and $G$ have finite dimensional values. We have the following generalization of corollary \ref{cor-per}.
\begin{corollary}\label{cor-gen}
If $G$ is homogeneous of degree $d$ then there is a finite number of graded vector spaces $E_\lambda$ indexed by partitions of $d$, such that for all $r\ge 0$
$$\Ext^*_{\PP_\kk}(F^{(r)},G^{(r)})= \bigoplus_{\mu\in \Lambda(d,p^r)} E_{\pi(\mu)} \left[\sum_{i=1}^{p^r} 2(i-1)\mu_i\right]\;.$$
\end{corollary}
Note that there is only a finite number of $E_{\lambda}$ appearing in this decomposition, since there is only a finite number of partitions $\lambda$ of $d$. The integer $r$ plays a role only for the number of factors $E_\lambda$ and the shift. One may draw interesting consequences of this. For example there is a numerical polynomial $f$ of degree $\le d$ depending of $F$, $G$ but not on $r$, such that for $r\ge \log_p(d)$ the total dimension of the graded vector space $\Ext^*_{\PP_\kk}(F^{(r)},G^{(r)})$ is equal to $f(p^r)$.

\subsection{More groups and more general coefficients}\label{subsubsec-more}
Although theorem \ref{thm-Tformula} is a simple and complete solution to problem \ref{mprob-bis}, the reader may feel unsatisfied because theorem \ref{thm-Tformula} adresses only a small part of the original problem \ref{mprob}. Indeed:
\begin{enumerate}
\item[(a)] it considers the group scheme $G=GL_n$ only,
\item[(b)] it considers polynomial representations only, 
\item[(c)] it considers only stable representations (i.e. when $n$ is big enough with respect to the degrees of the representations).
\end{enumerate}
It seems hard to improve on problem (c), for strict polynomial technology really relies on computational simplifications which are known to appear only in the stable range, i.e. when $n$ is big enough. But problems (a), (b)  can be successfully adressed, and we wish to present solutions to these problems here. 

We first consider problem (b). Theorem \ref{thm-Tformula} gives access to cohomology groups of the form
$$\HH^*(GL_n, \mathrm{Hom}_\kk(V,W)^{(r)})\simeq \Ext^*_{GL_n}(V^{(r)},W^{(r)}) $$
for polynomial representations $V$ and $W$. Many interesting and natural representations of $GL_n$ are not covered by this theorem. For example symmetric powers of the adjoint representation $S^d(\mathfrak{gl}_n)$ are not of the form $\mathrm{Hom}_\kk(V,W)$. A solution to cover more general coefficients by a functorial approach was found in \cite{FF}. The idea is to use strict polynomial \emph{bi}functors. A strict polynomial bifunctor of bidegree $(d,e)$ is a functor 
$$B:(\mathcal{V}_\kk^f)^{\mathrm{op}}\times \mathcal{V}_\kk^f\to \mathcal{V}_\kk$$
such that for all $V$ the functor $W\mapsto B(V,W)$ is strict polynomial of degree $e$, and for all $W$ the functor $V\mapsto B(V,W)$ is strict polynomial of degree $d$.
\begin{example}\label{ex-bif}
If $F$ and $G$ are strict polynomial functors of respective degrees $d$ and $e$, then $\mathrm{Hom}_\kk(F,G):(V,W)\mapsto \mathrm{Hom}_\kk(F(V),G(W))$ is a strict polynomial bifunctor of bidegree $(d,e)$. Let $\mathrm{gl}$ denote the strict polynomial bifunctor $(V,W)\mapsto \mathrm{Hom}_\kk(V,W)$ (of degree $(1,1)$). Then for all $F$ of degree $d$, the composite $F\circ \mathrm{gl}:(V,W)\mapsto F(\mathrm{Hom}_\kk(V,W))$ is a strict polynomial bifunctor of degree $(d,d)$.
\end{example}

We denote by $\PP_{\kk}(1,1)$ the category of strict polynomial bifunctors of finite bidegree (the notation $(1,1)$ suggests that the functors of this category have one contravariant variable and one covariant variable). If $B$ is a bifunctor, $GL_n$ acts on the vector space $B(\kk^n,\kk^n)$ by letting a matrix $g$ act as the endomorphism $B(g^{-1},g)$. The resulting representation is a \emph{rational} representation, but not a polynomial representation in general. We get in this way an exact functor
$$\mathrm{ev}_{(\kk^n,\kk^n)}:\PP_{\kk}(1,1)\to \text{Rational representations of $GL_n$}\;. $$
For example, $\mathrm{ev}_{(\kk^n,\kk^n)}(S^d\circ \mathrm{gl})$ is isomorphic to the symmetric power of the adjoint representation $S^d(\mathfrak{gl}_n)$. More generally, all finite dimensional rational representations of $GL_n$ lie in the image of this evaluation functor.
Given a strict polynomial bifunctor $B$, we denote by its cohomology $\HH_{\mathrm{gl}}^*(B)$ the extension groups:
$$\HH_{\mathrm{gl}}^*(B)=\bigoplus_{k\ge 0}\Ext^*_{\PP_\kk(1,1)}(\Gamma^k\circ \mathrm{gl},B)\;.$$
Franjou and Friedlander constructed \cite[Thm 1.5]{FF} a graded map, which is an isomorphism as soon as $n\ge \max\{d,e\}$, where $(d,e)$ is the bidegree of $B$:
$$ \HH^*_{\mathrm{gl}}(B) \to \HH^*(GL_n,B(\kk^n,\kk^n))\;. $$
Cohomology of strict polynomial bifunctors subsumes extensions of strict polynomial functors. Indeed, given strict polynomial functors $F$, $G$, there is an isomorphism \cite[Thm 1.5]{FF}:
$$\HH^*_{\mathrm{gl}}(\mathrm{Hom}_\kk(F,G))\simeq \Ext^*_{\PP_\kk}(F,G)\;.$$

\begin{remark}
An alternative construction of the isomorphism between bifunctor cohomology and $GL_n$-cohomology is given in \cite{TouzeClass}. As shown in \cite{TouzeHDR}, it also provides a graded isomorphism for $n\ge \max\{d,e\}$:
$$ \HH^*_{\mathrm{gl}}(B)\simeq \HH^*(SL_{n+1},B(\kk^{n+1},\kk^{n+1}))\;.$$
\end{remark}

Given a strict polynomial bifunctor $B$ we denote by $B^{(r)}$ the bifunctor obtained by precomposing each variable by the $r$-th Frobenius twist. That is, $B^{(r)}(V,W)=B(V^{(r)},W^{(r)})$, so that the rational $GL_n$-module $\mathrm{ev}_{(\kk^n,\kk^n)}(B^{(r)})$ is isomorphic to the $r$-th Frobenius twist of the rational $GL_n$-module $\mathrm{ev}_{(\kk^n,\kk^n)}(B)$. We also denote by $\Gamma^{d,E_r}$ the graded strict polynomial functor such that 
$$\Gamma^{d,E_r}(V)=\Gamma^d\mathrm{Hom}_\kk(E_r,V)= (\mathrm{Hom}_\kk(E_r,V)^{\otimes d})^{\Si_d}\subset \mathrm{Hom}_\kk(E_r,V)^{\otimes d}\;.$$
(The grading on $\Gamma^{d,E_r}$ is the one such that the inclusion above preserves gradings if the tensor product of graded vector spaces is defined as usual.) 
In \cite{TouzeUnivNew} we prove the following generalization of theorem \ref{thm-Tformula}.
\begin{theorem}\label{thm-TformulaBif}
For all strict polynomial bifunctors $B$, there is a graded isomorphism:
$$\HH_{\mathrm{gl}}^*(B^{(r)})\simeq \bigoplus_{k\ge 0}\Ext^*_{\PP_\kk(1,1)}(\Gamma^{k,E_r}\circ \mathrm{gl},B)\;.$$
In this isomorphism, the degree on the right hand side is understood as the total degree obtained by adding the $\Ext$-degree and the degree of the graded bifunctor 
$$\Gamma^{k, E_r}\circ \mathrm{gl}:(V,W)\mapsto \Gamma^k(\mathrm{Hom}_\kk(E_r,\mathrm{Hom}_\kk(V,W)))=\Gamma^k(\mathrm{Hom}_\kk(E_r\otimes V,W))\;.$$ 
As usual, $E_r$ is the graded vector space $\Ext^*_{\PP_\kk}(I^{(r)},I^{(r)})$ described in example \ref{ex-FScomputation}.
\end{theorem}
\begin{remark}
The statement given here is slightly different from \cite[Thm 1.2]{TouzeUnivNew} as the parameter $E_r$ is not at the same place in the formula. However, it is not hard to see that the two parametrizations used are adjoint, so that the two formulas are equivalent. See the end of section \ref{subsubsec-Step3} for more details on this topic.
\end{remark}

Next we consider problem (a). $\Ext$ computations in $\PP_\kk$ can also be used in order to compute $G$-cohomology when $G$ is a classical group scheme $Sp_{2n}$, $O_{n,n+\epsilon}$, or $SO_{n,n+\epsilon}$, with $\epsilon\in\{0,1\}$. We recall quickly the statements. Given strict polynomial functors $X$ and $F$ we denote by $\HH^*_{X}(F)$ the extension groups:
$$\HH^*_X(F)= \bigoplus_{k\ge 0}\Ext^*_{\PP_\kk}(\Gamma^k\circ X, F)\;.$$ 
In \cite[Thm 3.17]{TouzeClass} and \cite[Thm 7.24]{TouzeHDR} we proved the existence of a graded map, which is an isomorphism as soon as $2n\ge \deg F$:
$$ \HH^*_{\Lambda^2}(F) \to \HH^*(Sp_{2n},F(\kk^{2n\,\#}))\;. $$
Here $\kk^{2n\,\#}$ is the dual of the defining representation of $Sp_{2n}$.
Similary, by \cite[Thm 3.24]{TouzeClass} and \cite[Thm 7.24, Cor 7.31]{TouzeHDR} if $p\ne 2$, there are graded maps, which are isomorphisms as soon as $2n+\epsilon\ge \deg F+1$: 
$$ \HH^*_{S^2}(F) \to \HH^*(O_{n,n+\epsilon},F(\kk^{2n+\epsilon\,\#}))\xrightarrow[]{\mathrm{res}}\HH^*(SO_{n,n+\epsilon},F(\kk^{2n+\epsilon\,\#}))\;. $$ 
The analogue of theorem \ref{thm-Tformula} for extensions related to orthogonal and symplectic groups was proved by Pham Van Tuan. To be more specific, on can deduce the following statement from the results in \cite{Tuan}.
\begin{theorem}\label{thm-Pham}
If $X=S^2$ or $\Lambda^2$, and $p>2$, there are graded isomorphisms (where the grading on the right hand side is obtained by summing the $\Ext$ degree with the degree coming from the graded functor $\Gamma^{k,E_r}\circ X$):
$$\Ext^*_{\PP_\kk}(\Gamma^k\circ X,F^{(r)})\simeq \Ext^*_{\PP_\kk}(\Gamma^{k,E_r}\circ X,F)\;.$$
\end{theorem}

\subsection{Statements without functors}\label{subsubsec-nofct}
Problem \ref{mprob} was formulated in terms of the cohomology of group schemes, thus one might wish an answer in terms of the cohomology of group schemes. For the convenience of the reader, we now translate the main results obtained with functorial techniques into ready-to-use functor-free statements.
We use the following conventions and notations.
\begin{itemize}
\item
A representation of a classical group scheme $G\subset GL_n$ is called \emph{polynomial of degree $d$} if it may be written as the restriction to $G$ of a polynomial representation of degree $d$ of $GL_n$. 
\item
Let $G_n=GL_n$ or $SL_n$. A rational representation of $G_n\times G_m$ is said to be \emph{polynomial of bidegree $(d,e)$} if its restriction to $G_n$ (resp. $G_m$) is polynomial of degree $d$ (resp. $e$).
We let $\kk^n\boxtimes\kk^m$ be the polynomial representation of bidegree $(1,1)$ acted on by $G_n\times G_m$ via the formula $(g,h)\cdot v\otimes w= gv\otimes hw$.
\item Let $G_n=GL_n$ or $SL_n$. Using the embedding $G_n\to G_n\times G_n$, $g\mapsto (g^{-1},g)$, any polynomial representation $V$ of $G_n\times G_n$ yields a (non polynomial) representation of $G_n$, which we denote by $V_{\mathrm{conj}}$. For example, if $G_n=GL_n$ then $S^d(\kk^n\boxtimes\kk^n)_{\mathrm{conj}}$ equals $S^d(\mathfrak{gl}_n)$, the symmetric power of the adjoint representation of $GL_n$.
\item We define gradings on representations as follows.  We denote by $V[i]$ a copy of a representation, placed in degree $i$. Unadorned representations are placed in degree zero, i.e. $V=V[0]$. Also, the $\kk$-linear dual of a represention $V$ (acted on by $G$ via $(gf)(v)= f(g^{-1}v)$) is denoted by $V^{\#}$.
\end{itemize}
Then theorems \ref{thm-TformulaBif} and \ref{thm-Pham} can be reformulated as follows.
\begin{theorem}[Type A]
Let $G_n=SL_{n}$ or $GL_n$. Let $V$ be a polynomial representation of $G_n\times G_n$ of bidegree $(d,e)$. Let $r\ge 0$ and assume that $n\ge \max\{dp^r,ep^r\}+1$. There is a graded isomorphism (we take the total degree on the right hand side):
$$\HH^*(G_n,V_\mathrm{conj}^{(r)})\simeq \HH^*\left(G_n\times G_n\;,\; A_r^*\otimes V\right)\;,$$
where the graded representation $A_r^*$ denotes the symmetric algebra on the graded $G_n\times G_n$-representation $\bigoplus_{i=0}^{p^r-1}\kk^{n\#}\boxtimes \kk^{n\#}[2i]$.
\end{theorem}

\begin{theorem}[Type C]
Let $V$ be a polynomial representation of $Sp_{2n}$ of degree $d$. Let $r\ge 0$ and assume that  $2n\ge dp^r$, and $p$ is odd. There is a a graded isomorphism  (we take the total degree on the right hand side):
$$\HH^*(Sp_{2n},V^{(r)})\simeq \HH^*\left(GL_{2n}\;,\; C_r^*\otimes V\right)\;,$$
where  the graded representation $C_r^*$ denotes the symmetric algebra on the graded $GL_n$-representation $\bigoplus_{i=0}^{p^r-1}\Lambda^2(\kk^{2n\#})[2i]$.
\end{theorem}

\begin{theorem}[Types B and D]
Assume $p\ne 2$, let $G_{n,n+\epsilon}=O_{n,n+\epsilon}$ or $SO_{n,n+\epsilon}$ with $\epsilon\in\{0,1\}$, and let $V$ be a polynomial representation of $G_{n,n+\epsilon}$ of degree $d$. Let $r\ge 0$ and assume that $2n+\epsilon\ge dp^r+1$. There is a a graded isomorphism  (we take the total degree on the right hand side):
$$\HH^*(G_{n,n+\epsilon},V^{(r)})\simeq \HH^*\left(GL_{2n+\epsilon}\;,\; BD_r^*\otimes V\right)\;,$$
where the graded representation $BD_r^*$ denotes the symmetric algebra on the graded $GL_n$-representation $\bigoplus_{i=0}^{p^r-1}S^2(\kk^{2n+\epsilon\#})[2i]$.
\end{theorem}

\begin{remark}
If $r=0$ then $V^{(0)}=V$ and the results stated in this section can be interpreted as polynomial induction formulas. These formulas follow directly from the results of \cite{TouzeClass} (improved in \cite{TouzeHDR}). There is no such interpretation if $r\ge 1$, although the isomorphisms have  similarities with induction isomorphisms. 
\end{remark}

\section{Cohomology of twisted representations versus cohomology classes}\label{sec-eq}

As recalled in section \ref{subsec-Fin}, one of the crucial ingredients of the proof  \cite{TvdK} of cohomological finite generation for reductive group schemes is the construction of some cohomology classes living in $\HH^*(GL_n,\Gamma^d(\mathfrak{gl}_n)^{(1)})$, for $d\ge 0$ and $n\gg 0$. The purpose of this section is to show that this construction is equivalent to the solution of problem \ref{mprob} given in theorem \ref{thm-TformulaBif}.

It is not difficult to use theorem \ref{thm-TformulaBif} to construct the desired classes. This is explained in \cite{TouzeUnivNew}, and we briefly reiterate the argument in section \ref{subsec-sens-facile}.

Conversely, we explain how the cohomology classes can be used to prove theorem \ref{thm-TformulaBif} in section \ref{subsec-sens-dur}. The strategy given here differs from the arguments of \cite{TouzeUnivNew,Chalupnik3} in that it does not use the adjoint functor to the precomposition of the Frobenius twist (using this adjoint was a key idea for the earlier proofs, due to M. Cha{\l}upnik). Moreover, this proof can be adapted to cover theorem \ref{thm-Pham} related to the cohomology of orthogonal and symplectic types, without relying on the formality statements established in \cite{Tuan}. 
In order to deal with all settings simultaneously (i.e. symplectic, orthogonal and general linear cases), we first need to introduce a few notations.

\subsection{Some recollections and notations}\label{subsec-nota}
In the sequel, $\mathcal{F}$ stands for $\PP_\kk$ -- the category of strict polynomial functors of bounded degree over a field $\kk$ of characteristic $p$ recalled in section \ref{subsubsec-pol}, or its bifunctor analogue $\PP_\kk(1,1)$ recalled in section \ref{subsubsec-more}. While the letter `$\mathcal{F}$' suggests that we are working in a functor category, our functor categories are very similar to categories of representations of finite dimensional algebras or groups. They are $\kk$-linear abelian with enough injectives and projectives, $\mathrm{Hom}$s between finite functors (i.e. functors with values in finite dimensional $\kk$-vector spaces) are finite dimensional, all objects are the union of their finite subobjects, and so on. Actually much of what is explained below can be understood while thinking of $\mathcal{F}$ as a category of representations.  

In order to deal with degrees without refering to the number of variables of the functors, we say that a bifunctor of bidegree $(d,e)$ has (total) degree $d+e$.

We will often have to compose (bi)functors. To avoid cumbersome notations, the composition symbol `$\circ$' will systematically be omitted in the sequel. Moreover the $r$-th Frobenius twist functor $I^{(r)}$ will most often be denoted by `$^{(r)}$'. For example the composition 
$S^d\circ I^{(r)}\circ \mathrm{gl}$ will be denoted by $S^{d\,(r)}\mathrm{gl}$.
We will also use concise notations for tensor products of symmetric or divided powers. If $\mu=(\mu_1,\dots,\mu_n)$ is a composition of $d$, i.e. a tuple of nonnegative integers with sum $\sum \mu_i=d$, then we let 
$S^\mu=S^{\mu_1}\otimes\dots\otimes S^{\mu_n}$, and similarly for divided powers. Sometimes compositions will be written as a matrix of nonnegative integers $(\nu_{ij})$ with $1\le i\le \ell$, $0\le j\le m$, and we let $S^\nu$ be the tensor product $\bigotimes_{1\le i\le \ell}\bigotimes_{0\le j\le m}S^{\nu_{ij}}$.

We will constantly use the divided powers $\Gamma^dV=(V^{\otimes d})^{\Si_d}$, $d\ge 0$ of a vector space $V$. If $\mu$ is a composition of $d$, the tensor product of the inclusions $\Gamma^{\mu_i}\subset \otimes^{ \mu_i}$ yields an injective morphism $\Gamma^\mu \hookrightarrow \otimes^d$ which we will call the canonical inclusion. More generally, the canonical inclusion $\Gamma^\mu\hookrightarrow \Gamma^\nu$ refers to the unique (if is exist) morphism factoring the canonical inclusion $\Gamma^\mu\hookrightarrow \otimes^d$ through the canonical inclusion $\Gamma^\nu\hookrightarrow \otimes^d$. Recall also that $\bigoplus_{d\ge 0}\Gamma^d$ is a Hopf algebra (dual to the maybe more usual symmetric algebra). We will mainly use the coalgebra structure. The component $\Delta_{d,e}:\Gamma^{d+e}\to \Gamma^d\otimes\Gamma^e$ of the comultiplication is the canonical inclusion. Finally, divided powers satisfy a decomposition formula similar to that of symmetric powers, which we shall refer to as the \emph{exponential isomorphism for divided powers}:
$$\Gamma^{d}(V\oplus W)\simeq\bigoplus_{i+j=d}\Gamma^iV\otimes\Gamma^jW\;.$$

Fix an object $X$ of $\mathcal{F}$. Then we let 
$$\HH^*_{X}(F)=\bigoplus_{d\ge 0}\mathrm{Ext}^*_{\mathcal{F}}(\Gamma^dX,F)\;.$$
More generally, if we fix a graded vector space $E$ concentrated in even degrees and with finite total dimension, we can replace the functor $\Gamma^d$ by the graded functor $\Gamma^{d,E}$ defined by $\Gamma^{d,E}(V)=\Gamma^d\mathrm{Hom}_\kk(E,V)$.
Then for all $F$ in $\mathcal{F}$, we let 
$$\HH^*_{E, X}(F)=\bigoplus_{d\ge 0}\mathrm{Ext}^*_{\mathcal{F}}(\Gamma^{d,E}X,F)\;.$$
Note that this is a bigraded object: the first grading is the $\mathrm{Ext}$ grading, and the second grading is induced by the grading on $\Gamma^{d}\mathrm{Hom}_\kk(E,X)$. When refering to $\HH^*_{E, X}(F)$ as a graded vector space, we take the total grading. In particular if $\kk$ is considered as a graded vector space concentrated in degree zero, then $\HH^*_{\kk, X}(F)$ is graded isomorphic to $\HH^*_{X}(F)$.
These cohomology groups are equipped with a cup product:
\begin{align*}\mathrm{Ext}^i_{\mathcal{F}}(\Gamma^{d,E}X, F)\otimes \mathrm{Ext}^i_{\mathcal{F}}(\Gamma^{e,E}X, G)\xrightarrow[]{\cup}\mathrm{Ext}^{i+j}_{\mathcal{F}}(\Gamma^{d+e,E}X, F\otimes G)
\end{align*}
defined by $c\cup c'=\Delta_{d,e}^*(c\sqcup c')$, where `$\sqcup$' refers to the external cup product. The latter is the derived version of the tensor product, defined for all $F'$, $F$, $G'$ and $G$:
\begin{align*}\mathrm{Ext}^i_{\mathcal{F}}(F', F)\otimes \mathrm{Ext}^j_{\mathcal{F}}(G', G)\xrightarrow[]{\sqcup} \mathrm{Ext}^{i+j}_{\mathcal{F}}(F'\otimes G', F\otimes G)\;.\end{align*}
In the sequel, we will mainly use these definitions for the functors $X=S^2$, $X=\Lambda^2$ and for the bifunctor $X=\mathrm{gl}$. Note that in these cases $X$ commutes with Frobenius twists, i.e. we have isomorphisms: $^{(r)}X\simeq X^{(r)}$.

\subsection{The cohomology classes}\label{subsec-univ-classes}
In \cite{TouzeUniv,TvdK}, we constructed some cohomology classes living in the bifunctor cohomology groups $\HH^*_{\mathrm{gl}}(\Gamma^{*\,(1)}\mathrm{gl})$, or equivalently in the cohomology groups $\HH^*(GL_n,\Gamma^*(\mathfrak{gl}_n)^{(1)})$ for $n\gg 0$. As recalled in section \ref{subsec-Fin}, these classes play a role in the proof of cohomological finite generation for reductive groups. It is possible to construct analogous classes related to the cohomology of classical groups in types B,C,D. The uniform statement valid for all classical types is given by the next theorem.
\begin{theorem}\label{thm-univ-classes}
Assume that $(\mathcal{F},X)$ equals $(\mathcal{P}_\kk(1,1), \mathrm{gl})$ or $(\mathcal{P}_\kk,\Lambda^2)$, or $(\mathcal{P}_\kk,S^2)$. In the last two cases, assume moreover that $p$ is odd. There are graded $\kk$-linear maps, for $\ell> 0$:
$$\psi_\ell:\Gamma^{\ell}\big(\HH^*_X({}^{(1)} X)\big)\xrightarrow[]{} \HH^*_X(\Gamma^{\ell\,(1)} X)\;, $$
satisfying the following properties: 
\begin{enumerate}
\item $\psi_1: \HH^*_X({}^{(1)} X)\to \HH^*_X({}^{(1)} X)$ is the identity map,
\item for all positive $\ell,m$, the following diagrams of graded vector spaces and graded maps commutes:
$$\xymatrix{
\HH^*_X(\,\Gamma^{\ell+m\,(1)} X\,)\ar[rr]^-{(\Delta_{\ell,m})_*}&&\HH^*_X\big(\,(\Gamma^{\ell}\otimes\Gamma^m)^{(1)} X\,\big)\\
\Gamma^{\ell+m}\big(\,\HH^*_X({}^{(1)} X)\,\big)\ar[u]_-{\psi_{\ell+m}}\ar[rr]^-{\Delta_{\ell,m}}&& \Gamma^\ell\big(\,\HH^*_X({}^{(1)} X)\,\big)\,\otimes\, \Gamma^m\big(\,\HH^*_X({}^{(1)} X)\,\big)\ar[u]_-{\psi_{\ell}\cup \psi_m}
}\;,$$
$$\xymatrix{
\HH^*_X\big(\,(\Gamma^{\ell}\otimes\Gamma^m)^{(1)} X\,\big)\ar[rr]^-{\mathrm{mult}_*}&&\HH^*_X(\,\Gamma^{\ell+m\,(1)} X\,)\\
\Gamma^\ell\big(\,\HH^*_X({}^{(1)} X)\,\big)\,\otimes\, \Gamma^m\big(\,\HH^*_X({}^{(1)} X)\,\big)\ar[u]_-{\psi_{\ell}\cup \psi_m}\ar[rr]^-{\mathrm{mult}}&& \Gamma^{\ell+m}\big(\,\HH^*_X({}^{(1)} X)\,\big)\ar[u]_-{\psi_{\ell+m}}
}\;,$$
\end{enumerate}
\end{theorem}

\begin{remark}\label{rk-uni-versal}
For $(\mathcal{F},X)=(\mathcal{P}_\kk(1,1),\mathrm{gl})$, the classes in the image of $\psi_\ell$ are called `universal classes'. However it is not clear what universal property these classes satisfy. In particular, we don't know if the morphisms $\psi_\ell$  are uniquely determined. The problem of uniqueness is discussed in \cite[remark 4]{TouzeUnivNew}. The classes are `versal' in some sense though, since any affine algebraic group scheme $G$ is a subgroup of $GL_n$ for some $n$, hence receives by restriction some classes in $\HH^*(G,\Gamma^{d}(\mathfrak{g})^{(1)})$. 
\end{remark}

As already mentioned, a direct proof of theorem \ref{thm-univ-classes} when $(\mathcal{F},X)=\mathcal{P}_\kk(1,1)$ is given in \cite{TouzeUniv,TvdK}. 
The construction is combinatorial, that is, we construct an explicit resolution of the representation $\Gamma^d(\mathfrak{gl}_n)^{(1)}$. The whole resolution is complicated but there is a small part of the resolution which is easy to understand, and where we can easily construct nontrivial cocycles. The idea to construct this explicit resolution is as follows.
\begin{enumerate}
\item The strict polynomial bifunctor $\Gamma^{d\,(1)}\mathrm{gl}$ is too complicated, so we simplify the problem. Instead of constructing an injective resolution of $\Gamma^{d\,(1)}\mathrm{gl}$, we rather construct an injective resolution $J$ of the strict polynomial functor $\Gamma^{d\,(1)}$.
\item\label{itt-2} Having constructed $J$, we can evaluate on $\mathrm{gl}$ to obtain a resolution $J\mathrm{gl}$ of $\Gamma^{d\,(1)}\mathrm{gl}$. Now $J\mathrm{gl}$ is not an \emph{injective} resolution, nonetheless it is a resolution by $\HH^*_{\mathfrak{gl}}(-)$-acyclic objects. Thus it is perfectly qualified to compute cohomology.
\end{enumerate}
The category $\mathcal{P}_\kk(1,1)$ does not play a fundamental role in the construction of the classes. Indeed, the construction would work as well if we replaced it by the category of rational $GL_n$-modules, with $\mathrm{gl}$ replaced by the adjoint representation $\mathfrak{gl}_n$ and  $\HH^*_{\mathrm{gl}}(-)$ replaced by $H^*(GL_n,-)$. More generally the construction of the classes would work verbatim in any situation where one wants to compute some cohomology of the form $\HH^*(\Gamma^{d\,(1)}(X))$ provided this cohomology is equipped with cup products, and the objects $S^\mu(X)$ are $\HH^*(-)$-acyclic. In particular it works for $(\mathcal{F},X)=(\mathcal{P}_\kk,S^2)$ or $(\mathcal{P}_\kk,\Lambda^2)$ in odd characteristic (in the later cases $\HH^*_X(-)$-acyclicity follows from the Cauchy filtration of \cite[Thm III.1.4]{ABW}, see the proof of \cite[Lm 3.1]{TouzeUniv}.).

We now provide yet another approach to prove theorem \ref{thm-univ-classes} for $(\mathcal{F},X)=(\mathcal{P}_\kk,S^2)$ or $(\mathcal{P}_\kk,\Lambda^2)$ in odd characteristic, following the idea that the classes in $\HH^*_{\mathrm{gl}}(\Gamma^{\ell\,(1)} \mathrm{gl})$ are `versal' as mentioned in remark \ref{rk-uni-versal}. To this purpose, we use restriction maps 
$$\mathrm{res}_F:\HH^*_{\mathrm{gl}}(F\mathrm{gl})\to \HH^*_{X}(F X)$$
natural with respect to the strict polynomial functor $F$ and compatible with cup products. To be more specific, let $F$ be homogeneous of degree $d$, and consider an extension $e$ represented by (here we make apparent the variables $U$, $V$ of the strict polynomials bifunctors): 
$$0\to F(\mathrm{Hom}_\kk(U,V))\to \cdots \to \Gamma^d\mathrm{Hom}_\kk(U,V)\to 0\;.$$
By replacing $U$ by $V^*$ one gets an extension $e'$ of $F\,\otimes^2$ by $\Gamma^d\,\otimes^2$. If $\iota:X\leftrightarrows \otimes^2:\pi$ are the canonical inclusion and projections, the restriction map is then defined by
$$\mathrm{res}_F(e)=  (\Gamma^d\iota)^*(F\pi)_*(e')\;.$$
Next lemma collects some basic computations which can be found in the literature. 
The composition of the last three isomorphisms of lemma \ref{lm-calcul-prelim} is exactly the restriction map defined above, so that $\mathrm{res}_{I^{(1)}}$ is an isomorphism. Thus we can construct simply the maps $\psi_\ell$ for $(\mathcal{P}_\kk,X)$ by restricting the maps $\psi_\ell$ for $(\mathcal{P}_\kk(1,1),\mathrm{gl})$. That is, if we add decorations `$^X$' and `$^\mathrm{gl}$' to distinguish  the two cases we let:
$$\psi_\ell^X := \mathrm{res}_{\Gamma^{\ell\,(1)}}\circ \psi_\ell^{\mathrm{gl}}\circ \Gamma^\ell(\mathrm{res}_{I^{(1)}}^{-1})\;.$$

\begin{lemma}\label{lm-calcul-prelim}
Assume that $p$ is odd and $X=\Lambda^2$ or $S^2$. Let $\mathcal{P}_\kk(2)$ denote the category of strict polynomial functors with two covariant variables, and let $\boxtimes^2$ denote the bifunctor $\boxtimes^2(U,V)= U\otimes V$.
We have a chain of isomorphisms of graded vector spaces:
$$E_1\simeq\HH^*_{\mathrm{gl}}(\mathrm{gl}^{(1)})\simeq \mathrm{Ext}^*_{\mathcal{P}_\kk(2)}(\Gamma^p\,\boxtimes^2,\boxtimes^{2\,(1)})\simeq \HH^*_{X}(\otimes^{2\,(1)})\simeq  \HH^*_{X}(X^{(1)})\;.$$
\end{lemma}
\begin{proof}
The first isomorphism is Friedlander and Suslin's computation \cite[Thm 4.10]{FS} for $r=1$ (see example \ref{ex-FScomputation} for more details) translated in terms of bifunctors with \cite[Thm 1.5]{FF} (see more explanations in section \ref{subsubsec-more}). The second isomorphism is induced by the equivalence of categories between $\mathcal{P}_\kk(1,1)$ and $\mathcal{P}_\kk(2)$ given by dualizing the first variable of the bifunctors, i.e. it sends a bifunctor $B$ to the bifunctor $B'$ defined by $B'(U,V)=B(U^\sharp,V)$, where $^\sharp$ denotes $\kk$-linear duality. The third isomorphism is induced by evaluating both variables of the functor on the same variable $V$, and then pulling back by the map $\Gamma^{p}(\iota)$ where $\iota:X\to \otimes^2$ is the canonical inclusion. The fact that is is an isomorphism is (a very particular case) of the proof \cite[Thm 6.6]{TouzeClass}. Finally the last isomorphism is induced by the quotient map $\otimes^{2\,(1)}\to X^{(1)}$ and it is an isomorphism by \cite[Thm 6.6]{TouzeClass}.
\end{proof}

\subsection{The untwisting isomorphisms}
Recall from section \ref{subsubsec-pol} that $E_r$ denotes the graded vector space which equals $\kk$ in degrees $2i$ for $0\le i<p^r$ and zero in the other degrees. The notation $E_r$ reminds that this graded vector space is isomorphic to $\Ext^*_{\PP_\kk}(I^{(r)},I^{(r)})$, as initially computed by Friedlander and Suslin \cite{FS}.
The following statement is an abstract form of theorems \ref{thm-TformulaBif} and \ref{thm-Pham}, where we have made explicit the implicit naturalities and compatibilities with cup products (not all these properties are established in \cite{TouzeUnivNew,Chalupnik3,Tuan}).

\begin{theorem}\label{thm-untwist}
Assume that $(\mathcal{F},X)$ equals $(\mathcal{P}_\kk(1,1), \mathrm{gl})$ or $(\mathcal{P}_\kk,\Lambda^2)$, or $(\mathcal{P}_\kk,S^2)$. In the last two cases, assume moreover that $p$ is odd. 
For all positive $r$ there are isomorphisms of graded vector spaces (take the total degree on the left hand side):
$$\phi_F: \HH_{E_r, X}^*(F)\simeq\HH^*_{X}(F^{(r)})\;.$$
Moreover, $\phi_F$ is natural with respect to $F$ and commutes with cup products.
\end{theorem}

\begin{remark}
Theorem \ref{thm-untwist} does not hold for $(\mathcal{P}_\kk,\Lambda^2)$ or $(\mathcal{P}_\kk,S^2)$ in characteristic $2$. Indeed, in both cases for $F=I$ and $r=1$ the source of $\phi_F$ would be zero for degree reasons (there are no $\mathrm{Ext}$ between homogeneous functors of different degrees) while the target is nonzero as there are non split extensions $0\to I^{(1)}\to S^2\to \Lambda^2\to 0$ and $0\to I^{(1)}\to S^2\to \otimes^2\to S^2\to 0$.
\end{remark}

\subsection{From theorem \ref{thm-untwist} to theorem \ref{thm-univ-classes}}\label{subsec-sens-facile}

We now assume that theorem \ref{thm-untwist} holds. Following \cite{TouzeUnivNew}, we are going to prove theorem \ref{thm-univ-classes}. 

Cup products define graded maps $\cup:\HH^0_{E_1, X}(X)^{\otimes \ell}\to \HH^0_{E_1, X}(X^{\otimes\ell})$. Since the graded functor $\HH^0_{E_1, X}(-)$ is left exact, $\HH^0_{E_1, X}(\Gamma^{\ell}X)$ identifies with the graded subspace of $\HH^0_{E_1, X}(X^{\otimes\ell})$ of elements invariant under the action of the symmetric group $\Si_\ell$ (induced by letting $\Si_\ell$ permuting the factors of the tensor product $X^{\otimes\ell}$). Thus there is a unique graded map $\alpha_\ell$ making the following diagram commutative.
$$\xymatrix{
\HH^0_{E_1, X}(\Gamma^{\ell}X)\ar@^{^{(}->}[rr] &&\HH^0_{E_1, X}(X^{\otimes\ell})\\
&&\ar@{-->}[llu]^-{\alpha_\ell}\ar[u]^-{\cup}\Gamma^{\ell}\big(\,\HH^0_{E_1, X}(X)\,\big)
}$$
By construction $\alpha_1$ is the identity map of $\HH^0_{E_1, X}(X)$, and the maps $\alpha_\ell$ fit into commutative diagrams:
$$\xymatrix{
\HH^0_{E_1, X}(\,\Gamma^{\ell+m} X\,)\ar[rr]^-{(\Delta_{\ell,m})_*}&&\HH^0_{E_1, X}\big(\,(\Gamma^{\ell}\otimes\Gamma^m)X\,\big)\\
\Gamma^{\ell+m}\big(\,\HH^0_{E_1, X}(X)\,\big)\ar[u]_-{\alpha_{\ell+m}}\ar[rr]^-{\Delta_{\ell,m}}&& \Gamma^\ell\big(\,\HH^0_{E_1, X}(X)\,\big)\,\otimes\, \Gamma^m\big(\,\HH^0_{E_1, X}(X)\,\big)\ar[u]_-{\alpha_{\ell}\cup \alpha_m}
}\;,$$
$$\xymatrix{
\HH^0_{E_1, X}\big(\,(\Gamma^{\ell}\otimes\Gamma^m) X\,\big)\ar[rr]^-{\mathrm{mult}_*}&&\HH^0_{E_1, X}(\,\Gamma^{\ell+m} X\,)\\
\Gamma^\ell\big(\,\HH^0_{E_1, X}(X)\,\big)\,\otimes\, \Gamma^m\big(\,\HH^0_{E_1, X}( X)\,\big)\ar[u]_-{\psi_{\ell}\cup \psi_m}\ar[rr]^-{\mathrm{mult}}&& \Gamma^{\ell+m}\big(\,\HH^0_{E_1, X}(X)\,\big)\ar[u]_-{\psi_{\ell+m}}
}\;.$$

Now we can use theorem \ref{thm-untwist} to convert this rather trivial $\HH^0_{E_1, X}(-)$ construction into the sought after $\HH^*_{X}(-)$-construction. Indeed by theorem \ref{thm-untwist}, we have graded isomorphisms (the equality on the right comes from the fact that $X$ is an injective object of $\mathcal{F}$):
$$\HH^*_X(X^{(1)})\xrightarrow[\simeq]{\phi_X^{-1}} \HH^*_{E_1, X}(X)=\HH^0_{E_1, X}(X)\;. $$
Theorem \ref{thm-untwist} also yields graded monomorphisms, compatible with cup products:
$$\iota_\ell: \HH^0_{E_1, X}(\Gamma^{\ell}X)\to \HH^*_{X}(\Gamma^{\ell}X^{(1)})\;. $$
Since for our $X$ we have a canonical isomorphism $\Gamma^{\ell}X^{(1)}\simeq \Gamma^{\ell\,(1)}X$ we define graded maps $\psi_\ell$ satisfying the required properties by the formula:
$$\psi_\ell= \iota_\ell\circ \alpha_\ell\circ \Gamma^d(\phi_X^{-1})\;.$$

\subsection{From theorem \ref{thm-univ-classes} to theorem \ref{thm-untwist}: an overview}\label{subsec-sens-dur}\label{subsec-strateg}
Now we assume that theorem \ref{thm-univ-classes} holds. We are going to use it to prove theorem \ref{thm-untwist}. 
Recall that $E_r$ denotes the graded vector space with $E_r^{2i}=\kk$ for $0\le i<p^r$ and which is zero in other degrees. Let $E_r^{(1)}$ be the same vector space with grading multiplied by $p$, i.e. $(E_r^{(1)})^{2pi}=\kk$ for $0\le i<p^r$ and $E_r^{(1)}$ is zero in the other degrees. Then for all positive $r$ we have an isomorphism of graded vector spaces:
$$E_{r} \simeq E_{r-1}^{(1)}\otimes E_1\;.$$
In particular, for all $r\ge 1$ we can construct the isomorphism $\phi_F:\HH^*_{E_r, X}(F)\simeq \HH^*_X(F^{(r)})$ of theorem \ref{thm-untwist} as the composition of the chain of $r$ isomorphisms
$$\HH^*_{E_r, X}(F)\xrightarrow[\simeq]{}\HH^*_{E_{r-1}, X}(F^{(1)})\xrightarrow[\simeq]{}\cdots \xrightarrow[\simeq]{}\HH^*_{E_0,X}(F^{(r)})=\HH^*_{X}(F^{(r)})$$
which are provided by the next result.
\begin{theorem}\label{thm-untwist-bis}
Let $(\mathcal{F},X)$ equal $(\mathcal{P}_\kk(1,1), \mathrm{gl})$ or $(\mathcal{P}_\kk,\Lambda^2)$ or  $(\mathcal{P}_\kk,S^2)$. In the last two cases, assume moreover that $\kk$ has odd characteristic $p$. 
Let $E$ be a graded vector space of finite total dimension and concentrated in even degrees, and let $E^{(1)}$ denote the same vector space, with homothetic grading defined by $E^{i}=(E^{(1)})^{pi}$.
Then there is a graded isomorphism, natural with respect to $F$ and compatible with cup products:
$$\phi_F^{E}:\HH^*_{E^{(1)}\otimes E_1, X}(F)\simeq \HH^*_{E, X}(F^{(1)})\;.$$
\end{theorem}

Thus, to prove theorem \ref{thm-untwist} it suffices to prove theorem \ref{thm-untwist-bis}. In the remainder of section \ref{subsec-sens-dur}, we describe the strategy of the proof of theorem  \ref{thm-untwist-bis}. This strategy is fairly general, i.e. it is not properties specific to strict polynomial functors. The reader can safely imagine that $\mathcal{F}$ is the category of rational $GL_n$-modules, and replace $\HH^*_{E,X}(-)$ by $\bigoplus_{d\ge 0}\mathrm{Ext}^*_{GL_n}(\Gamma^{d}\mathrm{Hom}_\kk(E,\mathfrak{gl}_n),-)$. Functorial techniques are relegated to section \ref{subsec-pf}, in which we implement concretely the strategy.

Recall that any functor in $\mathcal{F}$ splits as a finite direct sum of homogeneous functors, so it suffices to prove theorem \ref{thm-untwist-bis} for  $F$ homogeneous. Moreover there are no nontrivial $\Ext$ between homogeneous functors of different degrees. Since $X$ is homogeneous of degree $2$, both $\HH^*_{E, X}(F^{(1)})$ and $\HH^*_{E^{(1)}\otimes E_1, X}(F)$ are zero unless $F$ has even degree, say $2d$.
Thus, to prove theorem \ref{thm-untwist}, we have to construct a graded isomorphism $\phi_F^E$, natural with respect to $F$ homogeneous of degree $2d$:
$$\underbrace{\HH^*_{E^{(1)}\otimes E_1, X}\left(F\right)}_{
\text{\small $
=\Ext^*_{\mathcal{F}}\left(\Gamma^{d}\mathrm{Hom}_\kk(E^{(1)}\otimes E_1,X)\,,\,F\right)
$}
}\xrightarrow[]{\qquad \text{\footnotesize $\phi_F^E$}\qquad} \underbrace{\HH^*_{E, X}\left(F^{(1)}\right)}_{
\text{\small
$=\Ext^*_{\mathcal{F}}\left(\Gamma^{dp}\mathrm{Hom}_\kk(E, X),F^{(1)}\right)
$}}\;.$$

Fix a basis $(b_i)_{1\le i\le \ell}$ of $E$ with the $b_i$ homogeneous, and a basis $(e_0,\dots,e_{p-1})$ of $E_1$, with $\deg e_j=2j$. Then, using the exponential isomorphism for divided powers, one may rewrite the domain and codomain of the sought-after $\phi_F^E$ as direct sums:
\begin{align*}
\HH^*_{E^{(1)}\otimes E_1, X}\left(F\right)&=\bigoplus_{\nu\in\Lambda(d,\ell,p)}\Ext^*_{\mathcal{F}}\left(\Gamma^{\nu}X\,,\,F\right)\,[ps(\nu)+t(\nu)]\;, \\
\HH^*_{E, X}\left(F^{(1)}\right)&=\bigoplus_{\mu\in \Lambda(pd,\ell,1)}\Ext^*_{\mathcal{F}}\left(\Gamma^{\mu},F^{(1)}\right)\,[s(\mu)]\;.
\end{align*}
In these decompositions, $\Lambda(a,\ell,m)$ denotes the set of matrices of nonnegative integers $(\nu_{ij})$ with $1\le i\le \ell$ and $0\le j<m$ such that $\sum_{i,j} \nu_{i,j}=a$, and the brackets indicate a shift of cohomological grading, for example $\Ext^{k}_{\mathcal{F}}\left(\Gamma^{\mu},F^{(1)}\right)[s(\mu)]$ is concentrated in degree $k+s(\mu)$. The shifts are weighted sums of the coefficients of the matrices defined by $s(\nu)=\sum_{i,j}\nu_{i,j}\deg(b_i)$ and $t(\nu)=\sum_{i,j} \nu_{i,j}\deg(e_j)$.

By summing the coefficients in each row of a matrix $\nu\in \Lambda(d,\ell,p)$ an multiplying the result by $p$, one obtains a matrix $\overline{\nu}\in \Lambda(pd,\ell,1)$. We construct the graded morphism $\phi_F$ as follows. For each $\nu\in \Lambda(d,\ell,p)$, we choose a cohomology class 
$$c_\nu\in \Ext^{t(\nu)}_{\mathcal{F}}(\Gamma^{\overline{\nu}}X,\Gamma^{\nu}X^{(1)})\;\subset\; \HH^{t(\nu)}_{E, X}(\Gamma^{\nu}X^{(1)})\;.$$ 
We define the restriction of $\phi_F$ to the summand of $\HH^*_{E^{(1)}\otimes E_1, X}(F)$ indexed by $\nu$ as the following composition:
$$\xymatrix{
\Ext^*_{\mathcal{F}}(\Gamma^\nu X,F)\; [ps(\nu)+t(\nu)]\ar@/_2pc/@{-->}[rrdd]_-{\phi_F}\ar[rr]^-{-^{(1)}}&&\Ext^*_{\mathcal{F}}(\Gamma^\nu X^{(1)},F^{(1)})\; [ps(\nu)+t(\nu)]\ar[d]^-{- \circ c_\nu}\\
&& \Ext^*_{\mathcal{F}}(\Gamma^{\overline{\nu}} X,F^{(1)})\;[s(\overline{\nu})]\ar@{^{(}->}[d]\\
&& \HH^*_{E, X}\left(F^{(1)}\right)\\
}$$
where `$\circ$' stands for the Yoneda splice of extensions.
At this point, we do not explain which classes $c_\nu$ we choose. However, whatever the choice of $c_\nu$ is, the following lemma is clear from the properties of Yoneda splices.
\begin{lemma}
The map $\phi^E:\HH^*_{E^{(1)}\otimes E_1, X}(-)\to \HH^*_{E, X}(-^{(1)})$ is a morphism of $\delta$-functors.
\end{lemma}

To prove that $\phi^E$ is an isomorphism of $\delta$-functors, it suffices to prove that $\phi_F^E$ is an isomorphism for all injectives $F$, by the following basic lemma.
\begin{lemma}\label{lm-basic}
Let $\mathcal{A}$ be an abelian category with enough injectives, let $\mathcal{B}$ be an abelian category, and let $S^*,T^*:\mathcal{A}\to \mathcal{B}$ be two $\delta$-functors satisfying $T^i=S^i=0$ for negative $i$.  Then a morphism of $\delta$-functors $\phi:S^*\to T^*$ is an isomorphism if and only if  $\phi_A:S^*(A)\to T^*(A)$ is an isomorphism for all injective objects $A$ in $\mathcal{A}$. 
\end{lemma}
Finally, it is claimed in theorem \ref{thm-untwist} that $\phi_F^E$ is compatible with cup products. Let $E'$ denote the tensor product $E^{(1)}\otimes E_1$. Then the homogeneous elements $b_i^{(1)}\otimes e_j$, $1\le i\le \ell$, $à\le j<p$ form a basis of $E'$. For all $d$, $e$ we have canonical decompositions (the first one has already been used in order to decompose the source of $\phi_F$):
\begin{align*}
&\Gamma^{d+e}\mathrm{Hom}_\kk(E', X^{(1)})\simeq \bigoplus_{
\text{\footnotesize $
\nu\in \Lambda(d+e,\ell,p)
$}} \Gamma^\nu X^{(1)}\;,\\
&\Gamma^d\mathrm{Hom}_\kk(E', X^{(1)})\otimes\Gamma^e\mathrm{Hom}_\kk(E', X^{(1)})\simeq \bigoplus_{\text{\footnotesize
$
\begin{array}{c}
\lambda\in \Lambda(d,\ell,p)\\
\mu\in \Lambda(e,\ell,p)
\end{array}
$
}}
\Gamma^\lambda X^{(1)}\otimes \Gamma^\mu X^{(1)} 
\end{align*}
We let $\Delta_{\lambda,\mu}^\nu:\Gamma^\nu X^{(1)}\to \Gamma^\lambda X^{(1)}\otimes\Gamma^\mu X^{(1)}$ be the components of the canonical inclusion $\Delta_{d,e}$
in this decomposition. Note that $\Delta_{\lambda,\mu}^\nu=0$ if $\nu\ne \lambda+\mu$. 
The next lemma is a formal consequence of naturality properties of cup products.
\begin{lemma}\label{lm-compat-classes}
The morphism $\phi_F^E$ is compatible with cup products if and only if for all matrices $\lambda$ and $\mu$, the following equality holds in $\HH^*_{E, X}(\Gamma^\lambda X^{(1)}\otimes\Gamma^\mu X^{(1)})$:
$$c_\lambda\cup c_\mu= \Delta_{\lambda,\mu\;*}^{\lambda+\mu}c_{\lambda+\mu}\;.$$
\end{lemma}

\subsection{From theorem \ref{thm-univ-classes} to theorem \ref{thm-untwist}: detailed proof}\label{subsec-pf}
In this section, we implement the strategy just described, namely we define the classes $c_\nu$ based on the classes provided by theorem \ref{thm-univ-classes} and then we prove that these classes have all the required properties. Checking the isomorphism on injectives requires classical computation techniques which are specific to `functor technology' (i.e. they have no full equivalent in the category of representations of groups). 

\subsubsection{Definition of the classes $c_\nu$}
Recall that $E_1$ is a graded vector space with homogeneous basis $(e_0,\dots,e_{p-1})$ with $\deg(e_j)=2j$. Lemma \ref{lm-calcul-prelim} yields a graded isomorphism between $E_1$ and $\HH^*_X(X^{(1)})$, and we still denote by $(e_0,\dots,e_{p-1})$ the corresponding homogeneous basis of $\HH^*_X(X^{(1)})$.
As there is a canonical isomorphism ${}^{(1)}X=X^{(1)}$, the graded morphisms $\psi_k$ of theorem \ref{thm-univ-classes} can be interpreted as a partial divided power structure on $\HH_X^{*}(\Gamma^{*}X^{(1)})$. We define `divided power classes' $\gamma^k(e_i)$ by letting $\gamma^0(e_i)=1\in \HH^0_X(\Gamma^{0}X^{(1)})=\kk$, and for $k>0$:
$$\gamma^k(e_i)=\psi_k(e_i^{\otimes k}) \quad\in \HH_X^{2ki}(\Gamma^{k}X^{(1)}) = \Ext^{2ki}_{\mathcal{F}}(\Gamma^{pk}X,\Gamma^{k}X^{(1)})\;.$$

For all matrices $\nu=(\nu_{i,j})\in \Lambda(d,\ell,p)$, there are canonical inclusions $\Delta_\nu:\Gamma^{\overline{\nu}}X\to \Gamma^{p\nu}X$ and we define the class 
$c_\nu\in \Ext^{t(\nu)}_{\mathcal{F}}(\Gamma^{\overline{\nu}}X,\Gamma^{\nu}X^{(1)})$
by the formula:
$$c_\nu= \Delta_\nu^*\left(\bigsqcup_{i,j}\gamma^{\nu_{i,j}}(e_j)\right)\;.$$
It is clear from the properties of the maps $\psi_\ell$ given in theorem \ref{thm-univ-classes} that these classes satisfy 
$c_\lambda\cup c_\mu= \Delta_{\lambda,\mu\;*}^{\lambda+\mu}c_{\lambda+\mu}$. Hence by lemma \ref{lm-compat-classes} the morphisms $\phi_F$ based on the classes $c_\mu$ are compatible with cup products. It remains to prove that the morphisms $\phi_F$ are isomorphisms. By lemma \ref{lm-basic}, it suffices to prove the isomorphism when $F$ is injective. We will proceed in several steps, starting with easy injectives (tensor product functors of low degree) and moving gradually towards general injectives. 

\subsubsection{The maps $\phi_F$ are isomorphisms (Step 1)}\label{subsubsec-Step1}
If $\mathcal{F}$ denotes the category of strict polynomial functors we let $Y=\otimes^2$ and if $\mathcal{F}$ denotes the category of strict polynomial bifunctors, we let $Y=\mathrm{gl}$. In both cases $Y$ is injective. As a first step towards the proof of theorem \ref{thm-untwist-bis}, we prove the following result.
\begin{proposition}\label{prop-step-1}
The map $\phi_Y^E$ is an isomorphism.
\end{proposition}

It is not hard to see that the vector space $\mathrm{Hom}_\mathcal{F}(X, Y)$ has dimension one, and we denote by $f$ a basis of it (we may take for $f$ the canonical inclusion $X\to Y$ in the functor case, and the identity map in the bifunctor case). Recall that we have fixed a homogeneous basis $(b_1,\dots,b_\ell)$ of $E$. We let $(b_1^{(1)},\dots,b_\ell^{(1)})$ be the corresponding homogeneous basis of $E^{(1)}$.  
By definition $\phi_Y^E$ is a map:
$$E^{(1)}\otimes E_1\otimes \mathrm{Hom}_\mathcal{F}(X, Y)\to \bigoplus_{\mu\in\Lambda(p,\ell,1)}\Ext^*_{\mathcal{F}}(\Gamma^{\mu}X,Y^{(1)})\;,$$
which sends each basis element $b_i^{(1)}\otimes e_j\otimes f$ of the source to the class
$$\phi_Y^E(b_i^{(1)}\otimes e_j\otimes f)= f_*(e_j)\in \Ext^*_{\mathcal{F}}(\Gamma^{(0,\dots,0,p,0,\dots,0)}X,Y^{(1)})$$
where $p$ occupies the $i$-th position in the $\ell$-tuple $(0,\dots,0,p,0,\dots,0)$. By lemma \ref{lm-calcul-prelim}, $\phi_Y^E$ induces an isomorphism onto the summand of the target indexed by the compositions $\mu$ of $p$ with exactly one nonzero coefficient. Thus, to finish the proof of proposition \ref{prop-step-1} it remains to prove that the other summands of the target are zero. This follows from the next lemma.
\begin{lemma}\label{lm-annul1}
Let $\mu$ be a composition of $p$ into $\ell$ parts. If $\mu$ has at least two nonzero coefficients then $\Ext^*_\mathcal{F}(\Gamma^\mu X,Y^{(1)})=0$.
\end{lemma}
\begin{proof}
If $k<p$ the canonical inclusion $\Gamma^k(V)\to V^{\otimes k}$ has a retract. Thus, under our hypotheses, $\Gamma^\mu X$ is a direct summand of $X^{\otimes p}$ which is itself a direct summand of $Y^{\otimes p}$. But $Y^{\otimes p}$ is projective so $\Ext^*_\mathcal{F}(Y^{\otimes p},Y^{(1)})$ is concentrated in degree zero. Finally, it is not hard to see (use e.g. \cite[Cor 2.12]{FS}) that there is no nonzero morphism $Y^{\otimes p}\to Y^{(1)}$. 
\end{proof}

\subsubsection{The maps $\phi_F$ are isomorphisms (Step 2)}\label{subsubsec-Step2} As a second step towards the proof of theorem \ref{thm-untwist-bis}, we consider the case of the injective functors $Y^{\otimes d}$, for $d\ge 1$.
\begin{proposition}\label{prop-step-2}
For all $d\ge 0$, the map $\phi_{Y^{\otimes d}}^E$ is an isomorphism.
\end{proposition}

In order to prove proposition \ref{prop-step-2}, we will rely on the following lemma.
\begin{lemma}\label{lm-strateg}
Let $F$ be an object of $\mathcal{F}$.  The following conditions are sufficient to prove that $\phi_{F^{\otimes d}}^E$ is an isomorphism:
\begin{enumerate}
\item $\phi_F^E$ is an isomorphism, 
\item\label{it-cond-3} the source and the target of $\phi_{F^{\otimes d}}^E$ have the same finite total dimension,
\item\label{it-cond-2} the following graded map is surjective:
$$
\begin{array}{ccc}
\HH^*_{E,X}(F^{(1)})^{\otimes d}\otimes \mathrm{End}_\mathcal{F}(F^{(1)\,\otimes d})&\to & \HH^*_{E, X}(F^{(1)\,\otimes d})\\
(c_1\otimes\cdots\otimes c_d)\otimes f & \mapsto & f_*(c_1\cup\cdots\cup c_d)
\end{array}\;.
$$
\end{enumerate} 
\end{lemma}
\begin{proof}
As $\phi^E$ is a natural transformation compatible with cup products, we have a commutative square (where the lower horizontal arrow is the map appearing in condition \eqref{it-cond-2}, and the upper horizontal arrow is defined similarly):
$$\xymatrix{
\HH^*_{E^{(1)}\otimes E_1, X}(F)^{\otimes d}\otimes \mathrm{End}_\mathcal{F}(F^{\otimes d})\ar[rr]\ar[d]^-{(\phi_F^E)^{\otimes d}\otimes\mathrm{Id}}_-{\simeq}&&\HH^*_{E^{(1)}\otimes E_1, X}(F^{\otimes d})\ar[d]^-{\phi_{F^{\otimes d}}^E}\\
\HH^*_{E, X}(F^{(1)})^{\otimes d}\otimes \mathrm{End}_\mathcal{F}(F^{(1)\,\otimes d})\ar@{->>}[rr]&&\HH^*_{E, X}(F^{(1)\,\otimes d})\;.
}$$
In particular $\phi_{F^{\otimes d}}^E$ is surjective. Condition \eqref{it-cond-3}  ensures it is an isomorphism.
\end{proof}

We have seen in section \ref{subsubsec-Step1} that the first condition of lemma \ref{lm-strateg} is satisfied for $F=Y$. We are now going to check conditions  \eqref{it-cond-3} and \eqref{it-cond-2}. We only prove the case $\mathcal{F}=\mathcal{P}_\kk$, the bifunctor case being similar\footnote{Moreover conditions \eqref{it-cond-3} and \eqref{it-cond-2} for bifunctors can also be deduced from computations already published in the literature. Indeed, for $E=\kk$ (concentrated in degree zero) the statement follows from \cite[Thm 1.8]{FF}, \cite[Thm Prop 5.4]{TouzeENS} or the computations of \cite[p. 781]{Chalupnik1}. For an arbitrary $E$ the computations can be deduced from the case $E=\kk$ by using the isomorphism
$\HH^*_{E,\mathrm{gl}}(B)\simeq \HH^*_{\mathrm{gl}}(B_E)$ explained at the end of section \ref{subsubsec-Step3}.}. 
For this computation, we momentarily (until the end of section \ref{subsubsec-Step2}) change our notations and indicate explicitly by the letter $V$ the variable of the strict polynomial functors, e.g. we write $\Gamma^d\mathrm{Hom}_\kk(E, X(V))$ instead of $\Gamma^d\mathrm{Hom}_\kk(E, X)$. We also use strict polynomial multifunctors of $2d$ (covariant) variables which we denote explicitly by $V_1,\dots,V_{2d}$. The category of multifunctors with $2d$ covariant variables is denoted by $\mathcal{P}_\kk(2d)$. The (derived) sum-diagonal adjunction (see e.g. the proof of \cite[Thm 1.7]{FFSS} or \cite[Section 5.3]{TouzeClass}) yields an isomorphism:
\begin{align*}\mathrm{Ext}^*_{\mathcal{P}_\kk(2d)}\Big(\Gamma^{pd}\mathrm{Hom}_\kk(E, X(&V_{1}\oplus\cdots\oplus V_{2d}))\;,\; V_1^{(1)}\otimes\cdots\otimes V_{2d}^{(1)}\Big)\\&\xrightarrow[\alpha]{\;\simeq\;} \mathrm{Ext}^*_{\mathcal{P}_\kk}\left(\Gamma^{pd}\mathrm{Hom}_\kk( E, X(V))\;,\; V^{(1)\,\otimes 2d}\right)\;.\end{align*}
To be more specific, $\alpha$ sends an extension $e_{V_1,\dots,V_{2d}}$ to the extension obtained by first replacing all the variables $V_i$ by $V$, and then pulling back by the map $\Gamma^{pd}(\delta)$, where $\delta:V\to V^{\oplus d}$ is the diagonal map which sends $v$ to $(v,\dots,v)$. 

Next we analyze the first argument of the source of $\alpha$. 
We consider the partitions of the set $\{1,\dots,2d\}$ into $d$ subsets of $2$ elements. Each partition of this kind can be uniquely represented as a $d$-tuple of pairs $((i_1,j_1),\dots,(i_d,j_d))$ of elements of $\{1,\dots,2d\}$ satisfying $i_n<j_n$ for all $n$ and $i_1<\dots<i_d$. We denote by $\Omega$ the set of such $d$-tuples. If $I\in \Omega$ we let $\Gamma^{p,I}(V_1,\dots, V_{2d})$ denote the multifunctor:
$$\Gamma^{p,I}(V_1,\dots,V_{2d})= \bigotimes_{1\le n\le d}\Gamma^p\left(V_{i_n}\otimes V_{j_n}\right)\;.$$
Using the exponential isomorphism for $X$ and for divided powers, we obtain that the graded multifunctor $\Gamma^{pd}\mathrm{Hom}_\kk(E, X(V_1\oplus\dots\oplus V_{2d}))$ is graded isomorphic to
$$\bigoplus_{I\in \Omega}\mathrm{Hom}_\kk(E^{(1)},\kk)^{\otimes d}\otimes\Gamma^{p,I}(V_1,\dots,V_{2d})\quad\oplus\;\text{other terms.}
$$
The `other terms' mentioned in the decomposition don't bring any contribution to the $\mathrm{Ext}$ by lemma \ref{lm-annul2} below, so that $\alpha$ actually induces a graded isomorphism:
\begin{align*}
\bigoplus_{
I\in \Omega
}
E^{(1)\,\otimes d}\otimes\mathrm{Ext}^*_{\mathcal{P}_\kk(2d)}\big(\Gamma^{p,I}(V_1,&\dots,V_{2d})\;,\; V_1^{(1)}\otimes\cdots\otimes V_{2d}^{(1)}\Big)\\&\xrightarrow[\alpha]{\;\simeq\;} \mathrm{Ext}^*_{\mathcal{P}_\kk}\left(\Gamma^{pd}\mathrm{Hom}_\kk( E, X(V))\;,\; V^{(1)\,\otimes 2d}\right)\;.\end{align*}
\begin{lemma}\label{lm-annul2}
For $1\le k\le N$ we consider a multifunctor $X_k(V_1,\dots,V_{2d})=V_{i_k}\otimes V_{j_k}$ with $1\le i_k,j_k\le d$, and a nonnegative integer $d_k$. We let 
$$F(V_1,\dots,V_{2d}):= \bigotimes_{1\le k\le N} \Gamma^{d_k}X_k(V_1,\dots,V_{2d})\;.$$
If
$\mathrm{Ext}^*_{\mathcal{P}_\kk(2d)}(F(V_1,\dots,V_{2d}), V_1^{(1)}\otimes \cdots\otimes V_{2d}^{(1)})$ is nonzero, then there exists a $d$-tuple $I\in\Omega$ such that the multifunctor $F(V_1,\dots,V_{2d})$ is isomorphic to $\Gamma^{p,I}(V_1,\dots,V_{2d})$.
\end{lemma}
\begin{proof}
Since there is no nonzero $\mathrm{Ext}$ between homogeneous multifunctors of different multidegrees, $F(V_1,\dots,V_{2d})$ must have multidegree $(p,\dots,p)$ in order that the $\mathrm{Ext}$ is nonzero. In particular, this implies that $d_k\le p$ for all $k$. We claim that the $\mathrm{Ext}$ is zero unless all the $d_k$ are equal to $p$. To this purpose, we follow the same strategy as in the proof of lemma \ref{lm-annul1}. Assume that $d_n<p$ for some $n$. Up to renumbering the variables we may assume that $X_n$ is non constant with respect to the variable $V_1$. Let $K\subset \{1,\dots,N\}$ the set of the indices $k$ such that $X_k$ is non constant with respect to $V_1$. As $F(V_1,\dots,V_{2d})$ has degree $p$ with respect to the variable $V_1$ we have $d_k<p$ for all $k\in K$ and in particular the canonical inclusion
$$\bigotimes_{k\in K}\Gamma^{d_k}X_{k}(V_1,\dots,V_{2d})\hookrightarrow
\bigotimes_{k\in K}X_{k}(V_1,\dots,V_{2d})^{\otimes d_k}$$
has a retract. Thus we can write $F(V_1,\dots,V_{2d})$ as a retract of some multifunctor of the form $V_{1}^{\otimes p}\otimes G(V_2,\dots,V_{2d})$. In particular by the K\"unneth formula this implies that $\mathrm{Ext}^*_{\mathcal{P}_\kk(2d)}(F(V_1,\dots,V_{2d}), V_1^{(1)}\otimes \cdots\otimes V_{2d}^{(1)})$
is isomorphic to the tensor product
$$\mathrm{Ext}^*_{\mathcal{P}_\kk}(V_1^{\otimes p}, V_1^{(1)})\otimes \mathrm{Ext}^*_{\mathcal{P}_\kk(2d-1)}(G(V_2,\dots,V_{2d}), V_2^{(1)}\otimes \cdots\otimes V_{2d}^{(1)})\;.$$
But the factor on the left is zero by lemma \ref{lm-annul1}. To sum up, we have proved that $F(V_1,\dots,V_{2d})$ must have multidegree $(p,\dots,p)$ and that all the $d_k$ are equal to $p$. This implies that $F(V_1,\dots,V_{2d})$ must be equal to some $\Gamma^{p,I}(V_1,\dots,V_{2d})$.
\end{proof}

Finally each term in the direct sum appearing at the source of $\alpha$ can be explicitly computed. Let $I=((i_1,j_1),\dots, (i_d,j_d))$ be an element of $\Omega$. Let $\sigma_I$ be the permutation of $\{1,\dots,2d\}$ defined by $\sigma_I(j_n)=2n$ and $\sigma_I(i_n)=2n-1$. There is a corresponding isomorphism of multifunctors (still denoted by $\sigma_I$):
$$\sigma_I: \bigotimes_{1\le n\le d} V_{i_n}^{(1)}\otimes V_{j_n}^{(1)}\xrightarrow[]{\simeq} \bigotimes_{1\le n\le d} V_{2n}^{(1)}\otimes V_{2n+1}^{(1)}\;. $$
If $T^{(1),I}(V_1,\dots,V_{2d})$ denotes the source of $\sigma_I$, we thus have a composite isomorphism (where the first isomorphism is provided by lemma \ref{lm-calcul-prelim} and where $\kappa$ refers to the K\"unneth map):
\begin{align*}
E_1^{\otimes d}\simeq \mathrm{\Ext}^*_{\mathcal{P}_\kk(2)}(&\Gamma^p(V_1\otimes V_2), V_1^{(1)}\otimes V_2^{(1)})^{\otimes d}\\&\xrightarrow[\simeq]{\kappa}
\mathrm{\Ext}^*_{\mathcal{P}_\kk(2d)}\left(\Gamma^{p,I}(V_1,\dots,V_{2d}), T^{(1),I}(V_1,\dots,V_{2d})\right)\\
&\xrightarrow[\simeq]{(\sigma_I)_*}\mathrm{\Ext}^*_{\mathcal{P}_\kk(2d)}\left(\Gamma^{p,I}(V_1,\dots,V_{2d}), V_1^{(1)}\otimes\cdots\otimes V_{2d}^{(1)}\right)\;.
\end{align*}
To sum up, we have constructed a completely explicit isomorphism:
\begin{align}
\bigoplus_{I\in \Omega} E^{(1)\,\otimes d}\otimes E_1^{\otimes d}\xrightarrow[]{\simeq}\mathrm{Ext}^*_{\mathcal{P}_\kk}\left(\Gamma^{pd}\mathrm{Hom}_\kk(E, X(V))\;,\; V^{(1)\,\otimes 2d}\right)\;. \label{eqn-iso-Ext}
\end{align}

By following the same reasoning, one gets a similar isomorphism (which could also be computed by more down-to-earth methods. Indeed $V^{\otimes 2d}$ is injective, hence this isomorphism is actually a mere $\mathrm{Hom}$ computation):
\begin{align}
\bigoplus_{I\in \Omega} E^{(1)\,\otimes d}\otimes E_1^{\otimes d}\xrightarrow[]{\simeq}\mathrm{Ext}^*_{\mathcal{P}_\kk}\left(\Gamma^{d}\mathrm{Hom}_\kk( E^{(1)}\otimes E_1, X(V))\;,\; V^{\otimes 2d}\right)\;. \label{eqn-iso-Ext2}
\end{align}
By comparing \eqref{eqn-iso-Ext} and \eqref{eqn-iso-Ext2}, we see that condition  \eqref{it-cond-3} of lemma \ref{lm-strateg} is satisfied.

To check condition \eqref{it-cond-2} of lemma \ref{lm-strateg}, we write down explicitly the effect of isomorphism \eqref{eqn-iso-Ext} on an element $x=b_{k_1}\otimes\cdots\otimes b_{k_n}\otimes a_{t_1}\otimes\cdots\otimes a_{t_n}$ belonging to the term $E^{(1)\,\otimes d}\otimes E_1^{\otimes d}$ indexed by $I$. To be more explicit, let $a'_{t_n}$ be the image of $a_{t_n}$ by the isomorphism $E_1\xrightarrow[]{\simeq} \mathrm{Ext}^*_{\mathcal{P}}(\Gamma^{p}X(V),V^{(1)\,\otimes 2})$ provided by lemma \ref{lm-calcul-prelim}.
If we define an $\ell$-tuple $\mu=(0,\dots,0,p,0,\dots,0)$ with `$p$' in $n$-th position, we can interpret $b_{k_n}\otimes a_{t_n}'$ as an element of 
$$\mathrm{Ext}^*_{\mathcal{P}}(\Gamma^{\mu}X(V),V^{(1)\,\otimes 2})\,[\deg(b_{k_n})]\; \subset\; \HH^*_{E, X(V)}(V^{(1)\,\otimes 2})$$
and by following carefully the explicit definition of isomorphism \eqref{eqn-iso-Ext}, we compute that isomorphism \eqref{eqn-iso-Ext} sends $x$ to
$$(\sigma_I)_*\left( (b_{k_1}\otimes a'_{t_1})\cup\cdots\cup (b_{k_d}\otimes a'_{t_d})\right) \in \HH^*_{E, X(V)}(V^{(1)\,\otimes 2d})\;.$$
As isomorphism \eqref{eqn-iso-Ext} is surjective, we obtain that condition \eqref{it-cond-2} of lemma \ref{lm-strateg} is satisfied.

\subsubsection{The maps $\phi_F$ are isomorphisms (Step 3)}\label{subsubsec-Step3}
To conclude the proof of theorem \ref{thm-untwist-bis}, we now prove that the maps $\phi_F$ are isomorphisms when $F$ is an arbitrary injective object of $\mathcal{F}$, homogeneous of degree $2d$. Since the source and the target of $\phi_F$ commute with arbitrary sums, we may restrict ourselves to proving this for an indecomposable injective $J$. If $\mathcal{F}=\PP_\kk$, any indecomposable injective $J$ homogeneous of degree $2d$ is a direct summand of a symmetric tensor $S^\mu$ for some composition $\mu$  of $2d$. If $\mathcal{F}=\PP_\kk(1,1)$, any indecomposable injective $J$ homogeneous of total degree $2d$ is a direct summand of a symmetric tensor $S^{\lambda,\mu}:(V,W)\mapsto S^\lambda(V^\sharp)\otimes S^\mu(W)$ for some tuples of nonnegative integers $\lambda=(\lambda_1,\dots,\lambda_s)$ and $\mu=(\mu_1,\dots,\mu_t)$ satisfying $\sum\lambda_i+\sum\mu_i=2d$ (and $V^\sharp$ is the $\kk$-linear dual of $V$). The next lemma records a consequence of this fact.
\begin{lemma}
The map $\phi_J$ is an isomorphism for all injectives $J$ if and only if it is an isomorphism for all symmetric tensors $J$.
\end{lemma}

Moreover, if $\mathcal{F}=\PP_\kk(1,1)$, we may assume that $\lambda$ and $\mu$ are both compositions of $d$, otherwise we have $\HH^*_X(S^{\lambda,\mu\;(r)})=0=\HH^*_{E_r, X}(S^{\lambda,\mu})$ for degree reasons (there are no nonzero $\mathrm{Ext}$ between homogeneous bifunctors of different bidegrees), so that $\phi_{S^{\lambda,\mu}}$ is trivially an isomorphism. 

Now symmetric tensors are quotients of $Y^{\otimes d}$. Indeed if $\mathcal{F}=\PP_\kk$, then $Y^{\otimes d}(V)=V^{\otimes 2d}$, the symmetric group $\Si_{2d}$ acts on $Y^{\otimes d}$ by permuting the variables, and for all compositions $\mu$ of $2d$ we have an isomorphism $(Y^{\otimes d})_{\Si_\mu}\simeq S^\mu$. Similarly in the bifunctor case $Y^{\otimes d}(V,W)= (V^\sharp)^{\otimes d}\otimes W^{\otimes d}$, the group $\Si_d\times \Si_d$ acts on $Y^{\otimes d}$ by permuting the variables, and we have an isomorphism:
$(Y^{\otimes d})_{\Si_\lambda\times\Si_\mu}\simeq S^{\lambda,\mu}$. Let $(J,\Si)$ denote $(S^{\lambda,\mu},\Si_\lambda\times\Si_\mu)$ (in the bifunctor setting) or $(S^\mu,\Si_\mu)$ (in the functor setting). By naturality of $\phi_F$, we have a commutative diagram 
$$\xymatrix{
\HH^*_{E, X}(Y^{\otimes d\,(1)})_{\Si}\ar[r]^-{}&\HH^*_{E, X}(J^{(1)})
\\
\HH^*_{E^{(1)}\otimes E_1, X}(Y^{\otimes d})_{\Si}\ar[u]^-{
\text{\normalsize $\phi_{Y^{\otimes d}}$}
}_-\simeq\ar[r]^-{}&\HH^*_{E^{(1)}\otimes E_1, X}(J)\ar[u]^-{
\text{\normalsize $\phi_{J}$}}
}\;.$$
The vertical map on the left is an isomorphism as we established it in section \ref{subsubsec-Step2}. 
The horizontal maps of the diagrams are isomorphisms by proposition \ref{prop-step-3} below, hence $\phi_J$ is an isomorphism for all symmetric tensors $J$, which finishes the proof of theorem \ref{thm-untwist}.

\begin{proposition}\label{prop-step-3}
Assume that $(\mathcal{F},X)=(\PP_\kk(1,1),\mathrm{gl})$. If $\lambda$ and $\mu$ are compositions of $d$, then for all $r\ge 0$ and all finite dimensional graded vector spaces $E$, the quotient map $Y^{\otimes d}\to S^{\lambda,\mu}$ induces an isomorphism
$$\HH^*_{E, X}(Y^{\otimes d\,(r)})_{\Si_\lambda\times \Si_\mu}\xrightarrow[]{\simeq}\HH^*_{E, X}(S^{\lambda,\mu\;(r)})$$
If $(\mathcal{F},X)=(\PP_\kk,S^2)$ or $(\PP_\kk,\Lambda^2)$, $p$ is odd and $\mu$ is a composition of $2d$ the quotient map 
$Y^{\otimes d}\to S^{\mu}$ induces an isomorphism
$$\HH^*_{E, X}(Y^{\otimes d\,(r)})_{\Si_\mu}\xrightarrow[]{\simeq}\HH^*_{E, X}(S^{\mu\;(r)})\;. $$
\end{proposition}

The remainder of the section is devoted to the proof of proposition \ref{prop-step-3}.
Assume first that $(\mathcal{F},X)=(\mathcal{P}_\kk(1,1),\mathrm{gl})$. If $E=\kk$ (concentrated in degree zero) then the statement follows from the results of \cite{FFSS}, or alternatively of  \cite{TouzeENS}. To be more specific, let $\mathbf{1}=(1,\dots,1)$ with `$1$' repeated $d$ times, hence $\otimes^d=S^{\mathbf{1}}$. The map $\HH^*_{\mathrm{gl}}(S^{\mu,\mathbf{1}\,(r)})\to \HH^*_{\mathrm{gl}}(S^{\mu,\lambda\,(r)})$ identifies \cite[Thm 1.5]{FF}
with the map $\Ext^*_{\PP_\kk}(\Gamma^{\lambda\,{(r)}},S^{\mathbf{1}\,(r)})\to \Ext^*_{\PP_\kk}(\Gamma^{\lambda\,{(r)}},S^{\mu\,(r)})$. The latter becomes an isomorphism after taking coinvariants under the action of $\Si_\mu$ at the source. This follows from \cite[Thm 4.5]{FFSS}, as explained in \cite{Chalupnik1}. An alternative proof without spectral sequences is given in \cite[Cor 4.7]{TouzeENS}. There is a similar result for the map $\HH^*_{\mathrm{gl}}(S^{\mathbf{1},\mathbf{1}\,(r)})\to \HH^*_{\mathrm{gl}}(S^{\mu,\mathbf{1}\,(r)})$. Thus we obtain the isomorphism in two steps:
$$\HH^*_{\mathrm{gl}}(S^{\lambda,\mu\,(r)})\simeq  \HH^*_{\mathrm{gl}}(S^{\lambda,\mathbf{1}\,(r)})_{\Si_\mu}\simeq (\HH^*_{\mathrm{gl}}(S^{\mathbf{1},\mathbf{1}\,(r)})_{\Si_\mu})_{\Si_\lambda}=\HH^*_{\mathrm{gl}}(S^{\mathbf{1},\mathbf{1}\,(r)})_{\Si_\lambda\times \Si_\mu}\;.$$

Now we prove the case of an arbitrary $E$. For this purpose we recall parametrizations of bifunctors by graded vector spaces. If $E$ is a graded vector space of finite total dimension, there is an exact lower parametrization functor \cite[Section 3.1]{TouzeUnivNew}
 $$\begin{array}{cccc}
-_E: &\PP_\kk(1,1)&\to & \PP_\kk(1,1)^*\\
& B & \mapsto & B_E
\end{array}$$ 
where $B_E$ is the bifunctor $(V,W)\mapsto B(V,(\bigoplus_{i} E^i)\otimes W)$, and the grading may be defined in the same fashion as in the functor case (explained just before theorem \ref{thm-Tformula}). Similarly, there is an exact upper parametrization functor 
 $$\begin{array}{cccc}
-^E: &\PP_\kk(1,1)&\to & \PP_\kk(1,1)^*\\
& B & \mapsto & B^E
\end{array}$$
where $B^E(V,W)=B(V,\mathrm{Hom}_\kk(\bigoplus_i E^i,W))$. Lower and upper parametrizations are adjoint, that is there is an isomorphism of graded $\kk$-vector spaces, natural with respect to $F$, $G$ and compatible with tensor products:
$$\mathrm{Hom}_{\mathcal{P}_\kk(1,1)}(F^E,G)\simeq \mathrm{Hom}_{\mathcal{P}_\kk(1,1)}(F,G_E)\;.$$
This property is easily checked when $F$ is a standard projective and $G$ is a standard injective by using the Yoneda lemma, and the general isomorphism follows by taking resolutions. Moreover as parametrization functors are exact, the adjunction isomorphism induces a similar adjunction on the $\Ext$-level. In particular, we have a graded isomorphism, natural with respect to $B$:
$$\HH^*_{E, \mathrm{gl}}(B)\simeq \HH^*_{\mathrm{gl}}(B_E)\;.$$
Since the bifunctor $(S^{\lambda,\mu\,(r)})_E$ splits as a direct sum of bifunctors of the form $S^{\nu,\mu\,(r)}$, it is now easy to prove that proposition \ref{prop-step-3} holds.   

Finally we prove the cases $(\mathcal{F},X)=(\PP_\kk,S^2)$ or $(\PP_\kk,\Lambda^2)$ and $p$ is odd. In these cases, $X$ is a direct summand of $\otimes^2$, hence the graded functor $\HH^*_{E, X}(-)$ is a direct summand of $\HH^*_{E, \otimes^2}(-)$. Hence it suffices to prove the isomorphism for $X=\otimes^2$. By using sum-diagonal adjunction and $\kk$-linear duality as in the proof of \cite[Thm 6.6]{TouzeClass} we obtain isomorphisms natural with respect to the functor $F$:
$$\HH^*_{E, \otimes^2}(F)\simeq \HH^*_{E, \mathrm{gl}}(F_{\boxplus})\;, $$
where $F_{\boxplus}$ denotes the bifunctor $(V,W)\mapsto F(V^\sharp\oplus W)$. But $S^{\mu\,(r)}_{\boxplus}$ decomposes as a direct sum of bifunctors of the form $S^{\lambda,\nu\,(r)}$. Thus the statement for $(\mathcal{F},X)=(\mathcal{P}_\kk,\otimes^2)$ can be deduced from the one for $(\mathcal{F},X)=(\PP_\kk(1,1),\mathrm{gl})$.

\end{document}